\documentclass[reqno]{amsproc}
\usepackage[utf8]{inputenc}
\usepackage{cite}
\usepackage{xcolor}
\usepackage{color}
\usepackage{calligra}
\usepackage[T1]{fontenc}
\usepackage{latexsym}
\usepackage[all]{xy}
\usepackage[normalem]{ulem}
\usepackage[pdftex]{hyperref}
\usepackage{hyperref}
\hypersetup{
    colorlinks=true,		% false: boxed links; true: colored links
  linkcolor=blue,		% color of internal links
  citecolor=blue,		% color of links to bibliography
  filecolor=magenta,		% color of file links
  urlcolor=blue,		% color of internet links
  bookmarksdepth=4
}
\usepackage{mathrsfs}
\usepackage{amsmath,amsfonts,amssymb,amsxtra}
\numberwithin{equation}{section}
\newtheorem{theorem}{\bf Theorem}[section]

\newtheorem{lem}{\bf Lemma}[section]
\newtheorem{cor}{\bf Corollary}[section]
\newtheorem{remark}{\bf Remark}[section]
\newtheorem{definition}{\bf Definition}[section]

\newtheorem{conj}{\bf Conjecture}[section]

\newcommand\dv{\mathrm{div}}
\newcommand\tr{\mathrm{tr}}

\usepackage{algorithmic}
\begin{document}

\title[Estimates of the gaps between consecutive eigenvalues]{Estimates of the gaps between consecutive eigenvalues for a class of elliptic differential operators in divergence form on Riemannian manifolds}

\author{Cristiano S. Silva$^1$}
\author{Juliana F.R. Miranda$^2$}
\author{Marcio C. Ara\'ujo Filho$^3$}

\address{$^{1,2}$Departamento de Matem\'atica, Universidade Federal do Amazonas, Av. General Rodrigo Oct\'avio, 6200, 69080-900 Manaus, Amazonas, Brazil.}
\address{$^3$Departamento de Matemática, Universidade Federal de Rondônia, Campus Ji-Paraná, R. Rio Amazonas, 351, Jardim dos Migrantes, 76900-726 Ji-Paraná, Rondônia, Brazil}

\email{$^1$cristianosilva@ufam.edu.br}
\email{$^2$jfrmiranda@ufam.edu.br}
\email{$^3$marcio.araujo@unir.br}

\urladdr{$^{1,2}$http://dmice.ufam.edu.br}
\urladdr{$^3$https://dmejp.unir.br}


\begin{abstract}
In this work, we obtain estimates for the upper bound of gaps between consecutive eigenvalues for the eigenvalue problem of a class of second-order elliptic differential operators in divergent form, with Dirichlet boundary conditions, in a limited domain of n-dimensional Euclidean space. This class of operators includes the well-known Laplacian and the square Cheng-Yau operator. For the Laplacian case, our estimate coincides with that obtained by D. Chen, T. Zheng, and H. Yang, which is the best possible in terms of the order of the eigenvalues. For pinched Cartan-Hadamard manifolds the estimates were made in particular cases of this operator.
\end{abstract}
\maketitle

\section{Introduction}

Let $(M^n,\langle\cdot,\cdot\rangle)$ be an $n$-dimensional complete Riemannian manifold, and $\Omega~\subset~M^n$ be a bounded domain with smooth boundary $\partial\Omega$. Let $T$ be a symmetric positive definite $(1,1)$-tensor on $M^n$ and a function $\eta\in C^2(M)$. Due to the boundedness
of $\Omega$ there exist two positive real constants $\varepsilon$ and $\delta$ such that $\varepsilon \leq\langle T(X), X \rangle\leq\delta$ for any unit vector field $X$ on $\Omega$.

In this paper, we are interested in studying the \textit{gap} between consecutive eigenvalues of the following eigenvalue problem with the Dirichlet boundary condition:
\begin{equation}\label{problem1}
    \left\{\begin{array}{ccccc}
    \mathcal{L}  u  &=& - \lambda u & \mbox{in } & \Omega,\\
     u&=&0 & \mbox{on} & \partial\Omega,
    \end{array}
    \right.
\end{equation}
where $\mathcal{L}$ is defined as the second-order elliptic differential operator in the $(\eta,T)$-divergence form
\begin{equation}\label{1.2}
    \mathcal{L}u =\dv_\eta (T(\nabla u)) := \dv(T(\nabla u)) - \langle \nabla \eta, T(\nabla u) \rangle.
\end{equation}
Here $\dv$ stands for the divergence of smooth vector fields and $\nabla$ for the gradient of smooth functions. We highlight that the operator $\mathcal{L}$  is a class of second-order elliptic differential operators in divergence form that includes, e.g., the Laplace-Beltrami and Cheng-Yau operators.

We can notice that $\mathcal{L}$ is a formally self-adjoint operator in the Hilbert space $\mathcal{H}_0^1(\Omega,e^{-\eta}d\Omega)$ of all functions in $L^2(\Omega,e^{-\eta}d\Omega)$ that vanish on $\partial \Omega$ in the sense of the trace. Then, Problem~\eqref{problem1} has a real and discrete spectrum
\begin{equation}\label{spectrum}
    0 < \lambda_1 \leq \lambda_2 \leq \cdots \leq \lambda_k \leq \cdots\to\infty,
\end{equation}
where each $\lambda_i$ is repeated according to its multiplicity. Eigenspaces belonging to distinct eigenvalues are orthogonal in $L^2(\Omega,e^{-\eta}d\Omega)$, which is the direct sum of all the eigenspaces. We refer to the dimension of each eigenspace as the multiplicity of the eigenvalue, for more details see Chavel~\cite{Isaac}.

For the case of $T$ being the identity operator and $\eta$ not necessarily constant function, the operator $\mathcal{L}$ becomes the drifted Laplace-Beltrami operator $\Delta_\eta$, that is, Problem~\eqref{problem1} becomes Dirichlet problem for operator  $\Delta_\eta$. In this case, it is worth mentioning the work by Gomes and Miranda~\cite{GM}, the paper by Xia and Xu~\cite{XiaXu}, and their references. In particular, when $T$ is the identity operator and $\eta$ is a constant function the operator $\mathcal{L}$ becomes the Laplace-Beltrami operator $\Delta$, that is, Problem~\eqref{problem1} becomes the Dirichlet problem for operator $\Delta$. For this case, we refer to Chen and Cheng~\cite{CCh}, Cheng and Yang~\cite{ChY1, ChY3}, Yang~\cite{Yang}, and their references.

When $T$ is divergence-free, that is, when $\dv T=0$ the third author and Gomes~\cite{AG} explained the relationship between the operator $\mathcal{L}$ and the operator $\square$ which was introduced by Cheng and Yau~\cite{ChYa}, that is, they checked that the operator $\mathcal{L}$ becomes
\begin{equation}
\nonumber
\mathcal{L}f = \square f - \langle\nabla \eta, T(\nabla f)\rangle :=\square_{\eta},
\end{equation}
where $f\in C^\infty(M)$, see Section~\ref{preliminaries}. Following the third author and Gomes, we call $\square_{\eta}$ the \emph{drifted Cheng-Yau operator} with a \emph{drifting function} $\eta$. In particular, if $\eta$ is constant, then $\square$ is the Cheng-Yau operator with $\dv T=0$.

As mentioned above, the purpose of this work was to obtain the gap estimates of the consecutive eigenvalues of Problem~\eqref{problem1} and we were motivated by the work of Chen, Zheng and Yang~\cite{Art}. In their work they proposed a conjecture for the gap between the consecutive eigenvalues of Laplace-Beltrami operator, cf. \cite[Conjecture~1.3]{Art}. Here we extend this conjecture for the case of  $\mathcal{L}$ operator, namely:
\begin{conj}\label{Cj2}
	Let $\Omega\subset M^n$ be a bounded domain in a complete Riemannian manifold. For Dirichlet Problem~\eqref{problem1}, the upper bound for the gap  between consecutive eigenvalues of the Laplace-Beltrami should be 
	\begin{equation}
		\nonumber
        \lambda_{k+1}-\lambda_{k}\leq C_{n,\Omega}k^{\frac{\delta}{n\varepsilon}}, \: k>1,
	\end{equation}
	where $C_{n,\Omega}$ is a constant dependent on $\Omega$, on the dimension $n$ de $M$, on the tensor $T$ and on the drifting function $\eta$.
\end{conj}

We prove the Conjecture~\ref{Cj2} for the case when $M$ is the Euclidean space. The proof is motivated by the corresponding results for the Laplacian operator $\Delta$ in \cite{Art}, see Theorem~\ref{T1.1}. When $M$ is the Hyperbolic space, we prove Conjecture~\ref{Cj2} with some additional hypothesis, we assume that the function $\eta$ is radially constant and the tensor $T$ is a limited $(1, 1)$-tensor such that $T(\nabla \ln{x_n}) = \psi\nabla\ln{x_n}$ for some radially constant $\psi\in C^{\infty}(M)$, see Theorem~\ref{T1.2}. We also prove this conjecture for the operator  $\square_{\eta}$ on bounded domain in pinched Cartan-Hadamard manifolds, see Theorem~\ref{T1.3}.

\begin{theorem}\label{T1.1}
	Let $\Omega \subset \mathbb{R}^{n}$ be a bounded domain in Euclidian space $\mathbb{R}^{n}$ and $\lambda_{k}$ be the k-th $(k>1)$ eigenvalue of the Dirichlet eigenvalue Problem~\eqref{problem1}. Then, we get
	\begin{equation}
		\nonumber
        \lambda_{k+1}-\lambda_{k}\leq C_{n,\Omega}k^{\frac{\delta}{n\varepsilon}},
	\end{equation}	
	where $\displaystyle C_{n,\Omega} = 4\left(\!\lambda_{1} + \dfrac{4C_0 + T_{0}^{2}}{4\delta}\right)\sqrt{\frac{\delta}{\sigma n}\!\left(1 + \frac{4\delta}{n\varepsilon}\right)}$, for $\sigma = 2\delta - \varepsilon$, $T_0=\sup_{\Omega}|\tr(\nabla T)|$ and
    $C_0=\sup_\Omega \left\{\frac{1}{2}\dv \Big( T\big(T(\nabla \eta)-\tr(\nabla T)\big)\Big) - \frac{1}{4}|T(\nabla \eta)|^2\right\}.$
\end{theorem}

Notice that for the case $\mathcal{L}=\square \quad \mbox{or} \quad \square_\eta$ we have $T_0=0$ and in the case of $\mathcal{L}=\Delta \quad \mbox{or} \quad \Delta_\eta $  we have, in addition, $\varepsilon=\delta = 1$. Moreover, in case when the function $\eta$ is constant we get $C_0 =0$. In particular we get the following corollary.

\begin{cor}\label{cor_1}
    Under the same setup as in Theorem~\ref{T1.1}, we have: 
	\begin{enumerate}
	\item[i)] For the drifted Cheng-Yau operator $\square_{\eta}$, we have
	\begin{equation}
		\nonumber
        \lambda_{k+1}-\lambda_{k}\leq C_{n,\Omega}k^{\frac{\delta}{n\varepsilon}},
	\end{equation}	
	where $\displaystyle C_{n,\Omega} = 4 \left(\lambda_{1} +  \frac{C_{0}}{\delta}\right)\sqrt{\!\frac{\delta}{\sigma n} {\left(\!1 \!+\! \frac{4\delta}{n\varepsilon}\!\right)}}$;
	\item[ii)] For the Cheng-Yau operator $\square$, we get
	\begin{equation}
        \nonumber
		\lambda_{k+1}-\lambda_{k}\leq C_{n,\Omega}k^{\frac{\delta}{n\varepsilon}},
	\end{equation}	
	where $C_{n,\Omega} = 4 \lambda_{1}\sqrt{\!\frac{\delta}{\sigma n} {\left(\!1 \!+\! \frac{4\delta}{n\varepsilon}\!\right)}}$;
	\item[iii)] For drifted Laplacian operator $\Delta_{\eta}$, we have
	\begin{equation}
		\nonumber
        \lambda_{k+1}-\lambda_{k}\leq C_{n,\Omega}k^{\frac{1}{n}},
	\end{equation}
	where $\displaystyle C_{n,\Omega}=4 \left(\lambda_{1} + C_{0}\right)\sqrt{\!\frac{1}{ n} {\left(\!1 \!+\! \frac{4}{n}\!\right)}}$;
	\item[iv)]  For the Laplacian operator $\Delta$, we get
	\begin{equation}
        \nonumber
        \lambda_{k+1}-\lambda_{k}\leq C_{n,\Omega}k^{\frac{1}{n}},
	\end{equation}	
	where $\displaystyle C_{n,\Omega} = 4 \lambda_{1}\sqrt{\!\frac{1}{ n} {\left(\!1 \!+\! \frac{4}{n}\!\right)}}$.
	\end{enumerate}
\end{cor}

\begin{remark}
	Item iv) of Corollary~\ref{cor_1} extend the result obtained by  Zheng and Yang~\cite[Theorem~1.5]{Art} for the case of a gap of the consecutive eigenvalues of the Dirichlet problem on $\mathbb{R}^{n}$ for Laplacian operator $\Delta$.
\end{remark}

With some additional hypotheses it was possible to prove  Theorems~\ref{T1.2} and \ref{T1.3}.

\begin{theorem}\label{T1.2}	
	Let $\Omega$ be a bounded domain in Hyperbolic space $\mathbb{H}^n(-1)$ and $\lambda_{k}$ be the k-th $(k>1)$ eigenvalue of the Dirichlet eigenvalue Problem~\eqref{problem1}. If the drifting function $\eta$  is radially constant and $T$ is an $(1,1)$-tensor, such that $T(\partial_{n})=\psi\partial_{n}$ for some function radially constant $\psi\in C^{\infty}(M)$, then
	\begin{equation}
        \nonumber
		\lambda_{k+1}-\lambda_{k}\leq C_{n,\Omega}k^{\frac{\delta}{n\varepsilon}},
	\end{equation}
	where $C_{n,\Omega}$ depends of the $\Omega$ and the dimension $n$ and is given by 
	\begin{equation}
        \nonumber
		C_{n,\Omega} = \frac{4}{\sqrt{\sigma}}\left[\left(1 + \frac{4\delta}{n\varepsilon}\right) \left(\!\delta\lambda_{1} \!-\! \dfrac{\varepsilon^{2}}{4}\!\big(n\!-\!1\big)^{2}\!\right) \left(\lambda_{1} + \frac{n^{2}H_{0}^{2} + 4C_{0} + T_{0}^{2}}{4\delta}\right) \right]^{\frac{1}{2}},
	\end{equation}
	where $\sigma$, $C_{0}$, $T_{0}$ are as in Theorem \ref{T1.1} and  $H_0=\sup_\Omega |{\bf H}_T|$ with ${\bf H}_T=\frac{1}{n}tr(\alpha\circ T)$, $\alpha$ being the second fundamental form.
\end{theorem}

From Theorem~\ref{T1.2} we get the following corollary.
\begin{cor}\label{cor_2}
	Under the same setup as in Theorem~\ref{T1.2}, we have: 
	\begin{enumerate}
	\item[i)] For the drifted Cheng-Yau operator $\square_{\eta}$, we have
	\begin{equation}
        \nonumber
		\lambda_{k+1}-\lambda_{k}\leq C_{n,\Omega}k^{\frac{\delta}{n\varepsilon}},
	\end{equation}	
	where $\displaystyle C_{n,\Omega} =\frac{4}{\sqrt{\sigma}} \left[\left(1 + \frac{4\delta}{n\varepsilon}\right) \left(\!\delta\lambda_1 \!-\! \dfrac{\varepsilon^2}{4}\!\big(n\!-\!1\big)^2\!\right) \left(\lambda_{1} + \frac{n^2H_{0}^2 + 4C_{0}}{4\delta}\right) \right]^{\frac{1}{2}}$;
	\item[ii)] For the Cheng-Yau operator $\square$, we get
	\begin{equation}
        \nonumber
		\lambda_{k+1}-\lambda_{k}\leq C_{n,\Omega}k^{\frac{\delta}{n\varepsilon}},
	\end{equation}	
	where $\displaystyle C_{n,\Omega} = \frac{4}{\sqrt{\sigma}} \left[\left(1 + \frac{4\delta}{n\varepsilon}\right) \left(\!\delta\lambda_1 \!-\! \dfrac{\varepsilon^2}{4}\!\big(n\!-\!1\big)^2\!\right) \left(\lambda_{1} + \frac{n^2H_{0}^2}{4\delta}\right) \right]^{\frac{1}{2}}$;
	\item[iii)] For drifted Laplace–Beltrami operator $\Delta_{\eta}$, we have
	\begin{equation}
	      \nonumber
        \lambda_{k+1}-\lambda_{k}\leq C_{n,\Omega}k^{\frac{1}{n}},
	\end{equation}	
	where $\displaystyle C_{n,\Omega} = 4\left[\left(1 + \frac{4}{n}\right) \left(\!\lambda_1 \!-\! \dfrac{1}{4}\!\big(n\!-\!1\big)^2\!\right) \left(\lambda_{1} + \frac{n^2H_{0}^2 + 4C_{0}}{4}\right) \right]^{\frac{1}{2}}$;
	\item[iv)] For the Laplace-Beltrami operator $\Delta$, we get
	\begin{equation}
        \nonumber
		\lambda_{k+1}-\lambda_{k}\leq C_{n,\Omega}k^{\frac{1}{n}},
	\end{equation}	
	where $\displaystyle C_{n,\Omega} = 4\left[\left(1 + \frac{4}{n}\right) \left(\!\lambda_1 \!-\! \dfrac{1}{4}\!\big(n\!-\!1\big)^2\!\right) \left(\lambda_{1} + \frac{n^2H_{0}^2}{4}\right)\right]^{\frac{1}{2}}$.
	\end{enumerate}
\end{cor}
	
\begin{remark}
	Item iv) of Corollary~\ref{cor_2} generalizes the result obtained by Chen, Zheng and Yang in \cite[Corollary~1.6]{Art} for the case of the gap for eigenvalues of Dirichlet problem on $\mathbb{H}^{n}$ for Laplace-Beltrami laplacian $\Delta$.
\end{remark}

For the next result, we should remember that an $n$-dimensional ($n\geq2$) simply connected smooth manifold with a geodesically complete Riemannian metric $\langle\cdot,\cdot\rangle$ of non positive sectional curvature has been called a Cartan-Hadamard manifold after the Cartan-Hadamard theorem, which says that a simply connected geodesically complete Riemannian manifold with nonpositive sectional curvatures $M^n$ has the same structure and topology of the Euclidean space $\mathbb{R}^{n}$. The most simple examples of Cartan-Hadamard manifold are Euclidean space and the Hyperbolic space, both with its respective canonical metrics.  More examples can be found in Shiga~\cite{Shi}.

We will call \textbf{pinched Cartan-Hadamard} manifold the Cartan-Hadamard manifolds that its sectional curvature $-\kappa_{1}^2\leq{K}\leq-\kappa_{2}^2$, for some constants $0\leq\kappa_{2}\leq\kappa_{1}$.

We say that a tensor $T$ is radially parallel if it is parallel in radial direction, that is, if $\nabla_{\partial r} T=0$. In particular, for each vector field $X$ such that $\nabla_{\partial r} X=0$ we say that it is a radially parallel field.
    
The next result proves Conjecture~\ref{Cj2}	for the case where $M$ is a pinched Cartan-Hadamard manifold with some hypothesis on $T$. This proof is inspired by the techniques used in the proofs of Proposition~2 in Fonseca and Gomes~\cite{FG} and Corollary~1.7 in Chen, Zheng and Yang~\cite{Art}.

\begin{theorem}\label{T1.3}	
    Let $M$ be a pinched $n$-dimensional Cartan-Hadamard manifold, $\Omega\subset M$ be a bounded domain and $\lambda_{k}$ the k-th $(k>1)$ eigenvalue of the \eqref{problem1} for operator $\square_{\eta}$, with $u_{k}$ its corresponding eigenfunction. Fix an origin $o\in{M}\setminus\overline{\Omega}$ and let $r(x)$ the distance function from $o$, such that $\partial_r$ is an eigenvector of $T$. If $T$ is radially parallel, then
	\begin{equation}
		\nonumber
        \lambda_{k+1}-\lambda_{k}\leq C_{n,\Omega}k^{\frac{\delta}{n\varepsilon}},
	\end{equation}
	where $C_{n,\Omega}$ depends of $\Omega$ and of the dimension $n$, which is given by
	\begin{equation}
    \nonumber
    \begin{split}
        C_{n,\Omega}=&\frac{4}{\sqrt{\sigma}}\left(\delta\lambda_{1}\!+\!\frac{[2(n-1)\delta^{2}-(2n-3) \varepsilon^{2}]\kappa_{1}^{2} - [n^2-2n+2] \varepsilon^{2}\kappa_{2}^{2}+2\delta^{2}\eta_{1}}{4}\right.\\
	    &\left.+\frac{\delta^2\eta_{r}(n-1)(\kappa_1 + \frac{1}{d})}{2}+ \frac{a(n,T)}{4d^{2}}\right)^{\!\!\frac{1}{2}}\!\!\!\left(1+\frac{4\delta}{n\varepsilon}\right)^{\!\!\frac{1}{2}}\!\!\left(\lambda_{1}\!+\!\frac{n^{2}H_{0}^{2}+4C_{0}}{4\delta}\right)^{\!\!\frac{1}{2}},
    \end{split}
	\end{equation}
	where $\eta_{1}=\sup_{\overline{\Omega}}|\nabla^{2}\eta(\partial{r},\partial{r})|$, $\eta_{r}=\max_{\overline{\Omega}}|\langle \nabla\eta,\partial_{r}\rangle|$, $d=dist(\Omega,o)$,
	\begin{equation}
    \nonumber
		a(n,T):=\left\{ \begin{array}{cl}
		0, & \mbox{ se  }\ 2(n-1)\delta^{2}-(n-1)^{2}\varepsilon^{2}\leq0, \\[5pt]
		2(n-1)\delta^{2}-(n-1)^{2}\varepsilon^{2}, & \mbox{ se }\ 2(n-1)\delta^{2}-(n-1)^{2}\varepsilon^{2} >0,
		\end{array}\right. 
	\end{equation}
	and $\sigma$, $H_{0}$, $C_{0}$ are as in Theorem \ref{T1.2}.
\end{theorem}

From Theorem~\ref{T1.3} we get the following corollary.
\begin{cor}\label{cor_3}
	Under the same setup as in Theorem~\ref{T1.3}, we have: 
	\begin{enumerate}
	\item[i)] For the Cheng-Yau operator $\square$, we get
	\begin{equation}
        \nonumber
		\lambda_{k+1}-\lambda_{k}\leq C_{n,\Omega}k^{\frac{\delta}{n\varepsilon}},
	\end{equation}	
	where 
    \begin{equation}
    \nonumber
    \begin{split}
        C_{n,\Omega}=\frac{4}{\sqrt{\sigma}}&\left(\delta\lambda_{1}\!+\!\frac{[2(n-1)\delta^{2}-(2n-3) \varepsilon^{2}]\kappa_{1}^{2} - [n^2-2n+2] \varepsilon^{2}\kappa_{2}^{2}}{4}\right.\\
	    &\left.\quad+\frac{a(n,T)}{4d^{2}}\right)^{\!\!\frac{1}{2}}\!\!\!\left(1+\frac{4\delta}{n\varepsilon}\right)^{\!\!\frac{1}{2}}\!\!\left(\lambda_{1}\!+\!\frac{n^{2}H_{0}^{2}}{4\delta}\right)^{\!\!\frac{1}{2}};
    \end{split}
    \end{equation}
	\item[ii)] For drifted Laplace–Beltrami operator $\Delta_{\eta}$, we have
	\begin{equation}
    \nonumber
	   \lambda_{k+1}-\lambda_{k}\leq C_{n,\Omega}k^{\frac{1}{n}},
	\end{equation}	
	where
    \begin{equation}
    \nonumber
    \begin{split}
        C_{n,\Omega}=4&\left(\lambda_{1}\!+\!\frac{\kappa_{1}^{2} - [n^2-2n+2]\kappa_{2}^{2}+2\eta_{1}}{4}+\frac{\eta_{r}(n-1)(\kappa_1 + \frac{1}{d})}{2}\right.\\
	    &\left.\quad+ \frac{a(n)}{4d^{2}}\right)^{\!\!\frac{1}{2}}\!\!\!\left(1+\frac{4}{n}\right)^{\!\!\frac{1}{2}}\!\!\left(\lambda_{1}\!+\!\frac{n^{2}H_{0}^{2}}{4}\right)^{\!\!\frac{1}{2}},
    \end{split}
	\end{equation}
    with $a(n):=\left\{ 
    \begin{array}{cl}
		0, & \mbox{ se  }\ n\geq3,\\[5pt]
		(n-1)(3-n), & \mbox{ se }\ n<3;
	\end{array}\right.$
	\item[iii)] For the Laplace-Beltrami operator $\Delta$, we get
	\begin{equation}
		\nonumber
        \lambda_{k+1}-\lambda_{k}\leq C_{n,\Omega}k^{\frac{1}{n}},
	\end{equation}	
	where 
    \begin{equation}
		\nonumber
        C_{n,\Omega}=4\!\left(\lambda_{1}\!+\!\frac{\kappa_{1}^{2} - [n^2-2n+2]\kappa_{2}^{2}}{4}+ \frac{a(n)}{4d^{2}}\right)^{\!\!\frac{1}{2}}\left(1+\frac{4}{n}\right)^{\!\!\frac{1}{2}}\!\!\left(\lambda_{1}\!+\!\frac{n^{2}H_{0}^{2}}{4}\right)^{\!\!\frac{1}{2}},
  \end{equation}
    with $a(n):=\left\{ 
    \begin{array}{cl}
		0, & \mbox{ se  }\ n\geq3, \\[5pt]
		(n-1)(3-n), & \mbox{ se }\ n<3.
	\end{array}\right.$
	\end{enumerate}
\end{cor}

\begin{remark}
	Item iii) of Corollary~\ref{cor_3} generalizes the result obtained by Chen, Zheng and Yang in \cite[Corolário~1.7]{Art} for the case of the gap for eigenvalues of Dirichlet problem on a pinched $n$-dimensional ($n\geq3$) Cartan-Hadamard manifold for Laplace-Beltrami laplacian $\Delta$.
\end{remark}

\section{Preliminaries}\label{preliminaries}

Throughout this paper, we are frequently using the identification of a $(0,2)$-tensor $T:\mathfrak{X}(M)\times\mathfrak{X}(M)\to~C^{\infty}(M)$ with its associated $(1,1)$-tensor $T:\mathfrak{X}(M)\to\mathfrak{X}(M)$ by the expression
\begin{equation}
    \nonumber
    \langle T(X), Y \rangle = T(X, Y).
\end{equation}
Moreover, the tensor $\langle\cdot,\cdot\rangle$ will be identified with the identity $I$ in $\mathfrak{X}(M)$. Since, 
\begin{equation}\label{2.1}
     \varepsilon \leq\langle T(X), X \rangle\leq\delta, 
\end{equation}   
for any unit vector field $X$ on $\Omega$, we have
\begin{equation}\label{T-property}
    \varepsilon  \langle T(Y), Y\rangle \leq |T(Y)|^2 \leq \delta  \langle T(Y), Y\rangle \quad \mbox{for all} \quad Y \in \mathfrak{X}(M),
\end{equation}
consequently, 
\begin{equation}
    \nonumber
    \varepsilon^2|\nabla \eta|^2 \leq |T(\nabla\eta)|^2 \leq \delta^2|\nabla\eta|^2.
\end{equation}
	
If $T$ is a symmetric positive definite  $(1,1)$-tensor on $M$, we can see that the tensor $\nabla_{X}T$ also is for each $X\in\mathfrak{X}(M)$, that is, 
\begin{equation}
    \nonumber
	\left\langle (\nabla_{X}T)Y, Z\right\rangle = \left\langle Y,  (\nabla_{X}T)Z\right\rangle.
\end{equation}
Moreover, we define the vector field $\mathrm{tr}(\nabla T)\in\mathfrak{X}(M)$ by
\begin{equation}
    \nonumber
	\mathrm{tr}(\nabla T) := \sum_{j=1}^{n}\nabla T(e_j,e_j) = \sum_{j=1}^{n}\left(\nabla_{e_{j}}T(e_j) - T(\nabla_{e_{j}}e_j)\right),
\end{equation}	
where $\{e_1,\cdots,e_n\}$ is an orthonormal local frame in $p\in M^n$.
Given the $(1,1)$-tensors $T$ and $S$  in $M^n$ and their respective adjoints $T^{*}$ and $S^{*}$ we recall that the Hilbert-Schmidt inner product is defined by
\begin{equation}
    \nonumber
	\langle T,S \rangle := \mathrm{tr}(TS^*) = \sum_{j=1}^{n} \langle TS^*(e_j),e_j \rangle = \sum_{j=1}^{n}\langle S^*(e_j),T^{*}(e_j) \rangle = \sum_{j=1}^{n}\langle T(e_j),S(e_j) \rangle.
\end{equation} 
	
\begin{definition}
	Let $T$ an $(1,1)$-tensor in $M^n$, let us define the divergence of the $T$ as being $(0,1)$-tensor given by
	\begin{equation}
        \nonumber
		(divT)(v)_{p} = \mathrm{tr}(w \mapsto (\nabla_w T)(v)_{p}),
	\end{equation}
	where $p\in{M}$ and $v\in{T_{p}M}$. Moreover, we say that the tensor $T$ is divergence-free when $divT=0$.
\end{definition}
	
Another important property is the following 
\begin{equation}\label{2.3}
	(divT)(Z) = div(T(Z)) - \left\langle \nabla Z , T^{*}\right\rangle,
\end{equation}
for $Z\in\mathfrak{X}(M)$.
 
Recall that, the Cheng-Yau operator is given by
\begin{equation}\label{SQ}
	\square f = \mathrm{tr}(\nabla^{2} f\circ T) = \left\langle \nabla^{2} f, T \right\rangle,
\end{equation}
where $f\in\mathbb{C}^{\infty}(M)$ and $T$ is a symmetric $(1,1)$-tensor. When $M^n$ is compact and orientable, they proved that the operator $\square$ is self-adjoint if, only if, $T$ is divergence-free.

From \eqref{1.2}, \eqref{2.3} and \eqref{SQ} we observe that
\begin{equation}\label{2.5}
   \mathcal{L}f = \square f - \left\langle div_{\eta}T, \nabla f \right\rangle =  \square_{\eta} f - \left\langle div T, \nabla f \right\rangle,
\end{equation}
where $\square_{\eta}f=\left\langle\nabla^{2}f,T\right\rangle-\left\langle\nabla\eta,T(\nabla f)\right\rangle$. In particular, if $\square$ operator is self-adjoint the equation \eqref{2.5} becomes
\begin{equation}
    \nonumber
	\mathcal{L}f = \square_{\eta} f = \left\langle \nabla^{2} f, T \right\rangle - \left\langle \nabla \eta, T(\nabla f) \right\rangle ,
\end{equation}
which is a first-order perturbation of the Cheng-Yau operator. 
	
One of the most important tool to obtain lower estimates for the first positive eigenvalue of the operator $\mathcal{L}$ in compact Riemannian manifolds is a Bochner type formula proven by Gomes and Miranda~\cite{GM} which is given by
\begin{equation}\label{Bochner}
	\frac{1}{2}\mathcal{L}(|\nabla{f}|^2) =\langle\nabla(\mathcal{L}f),\nabla{f}\rangle+R_{\eta,T}(\nabla{f},\nabla{f})+\langle\nabla^{2}f,\nabla^{2}f\circ{T}\rangle-\langle\nabla^{2}f,\nabla_{\nabla{f}}T\rangle,
\end{equation}
where $R_{\eta,T}=R_{T}-\nabla(div_{n}T)^{\sharp}$ with $R_{T}(X,Y):=\mathrm{tr}(T\circ(Z\mapsto{R}(X,Z)Y))$ and $R(X,Z)Y$ is the Riemannian curvature tensor $\langle\cdot,\cdot\rangle$.
	
The Bochner type formula~\eqref{Bochner} is the most general case of traditional Bochner formula with relates harmonic functions $f$ and the Ricci curvature. We can see particular cases of Equation~\eqref{Bochner} in the works of Bochner~\cite{B} (for Laplacian case), Setti~\cite[Proposition~2.1]{S} (for drifted Laplacian case) and Alencar et. al~\cite[Theorem~1.1]{AGD} or \cite[Equation~(2.1)]{GM} (for Cheng-Yau operator case).

Moreover, for drifted Cheng-Yau operator $\square_{\eta}$ from \eqref{Bochner} we have
\begin{equation}\label{Bo}
	\frac{1}{2}\square_{\eta}(|\nabla{f}|^2) =\langle\nabla(\square_{\eta}{f}),\nabla{f}\rangle+R_{\eta,T}(\nabla{f},\nabla{f})+\langle\nabla^{2}f,\nabla^{2}f\circ{T}\rangle-\langle\nabla^{2}f,\nabla_{\nabla{f}}T\rangle.
\end{equation}

Since $(M^n,\langle\cdot,\cdot\rangle)$ is an $n$-dimensional complete Riemannian manifold and $\Omega$ is a bounded domain with smooth boundary $\partial\Omega$, let us consider on $M$ the weighted measure given by $\mathrm{dm}=e^{-\eta}dvol\Omega$, for some smooth function $\eta$. 
It is known that the divergence theorem is valid for operator $\mathrm{div}_{\eta}$ with respect the weighted measure $\mathrm{dm}$, that is, we can get for operator $\mathcal{L}$ defined by \eqref{1.2}
\begin{equation}
    \nonumber
	\int_{\Omega}\mathcal{L}f \mathrm{dm} = \int_{\partial\Omega} \langle\nu,T(\nabla f)\rangle d\tau
\end{equation}
and
\begin{equation}\label{partes}
	\int_{\Omega}h\mathcal{L}f \mathrm{dm} = - \int_{\Omega}\left\langle T(\nabla h), \nabla f\right\rangle \mathrm{dm} + \int_{\partial\Omega} h\left\langle\nu, T(\nabla f)\right\rangle d\tau,
\end{equation} 
where $\nu$ is the outward unit normal vector field of $\partial \Omega$ and $d\tau=e^{-\eta}d\partial\Omega$ is the weighted volume form on the boundary.

From \eqref{partes} we can see that $\mathcal{L}$ is a symmetric operator on the functions space $C_{c}^{\infty}(\Omega)$, and Problem~\eqref{problem1} has a real and discrete spectrum as in \eqref{spectrum}. The eigenfunctions $u_{j}$ associated with the eigenvalues $\lambda_{j}$ form an orthonormal base on $L^{2}(\Omega,\mathrm{dm})$, then for $f\in L^{2}(\Omega,\mathrm{dm})$ we have
\begin{equation}\label{2.9}
	f = \sum_{j=1}^{\infty} \langle f,u_j\rangle u_j
\end{equation}
and
\begin{equation}\label{2.10}
	\|f\|^2 = \sum_{j=1}^{\infty} \langle f,u_j\rangle^2.
\end{equation}
The formulas \eqref{2.9} and \eqref{2.10} are called the Parseval identities. Finally, the eigenvalues of Problem~\eqref{problem1} can be obtained by the following expression
\begin{equation}
	\nonumber
    \lambda_{j} = -\int_{\Omega}u_{j}\mathcal{L}{u_{j}}\ \mathrm{dm} = \int_{\Omega}T(\nabla{u_{j}},\nabla{u_{j}})\ \mathrm{dm}.
\end{equation}

To demonstrate one of the auxiliary lemmas, the following procedure will be necessary. Given a real Riemannian manifold $(M,\langle\cdot,\cdot\rangle)$ we can give a complex structure to the tangent space for all $p\in M$ by 
\begin{equation}
    \nonumber
    T_p^{\mathbb{C}}M := \left\{X_p + iY_p ; \ X_p,Y_p\in T_pM \right\},
\end{equation}
and we provide $T_p^{\mathbb{C}}M$ with a Hermitian inner product defined through $\langle\cdot,\cdot\rangle_p$ making the following extension for it 
\begin{equation}
    \nonumber
	\langle iX,Y \rangle_p  = \langle X,iY \rangle_p = i\langle X,Y \rangle_p, \mbox{ where } \  {i}=\sqrt{-1} ,
\end{equation}
and defining
\begin{equation}
\nonumber
\begin{split}
    \langle\!\langle\cdot,\cdot\rangle\!\rangle_p :\ &T_p^{\mathbb{C}}M\times T_p^{\mathbb{C}}M \longrightarrow \quad \mathbb{C} \\
	&\quad \ (Z,W) \quad \ \longmapsto  \langle\!\langle Z,W \rangle\!\rangle_p := \langle Z, \overline{W} \rangle_p,
\end{split} 
\end{equation}
where $\overline{W}_p=X_p-iY_p$ is the conjugate vector of $W_p=X_p+iY_p$.
In a similar way, the space of differentiable vector fields will be given by
\begin{equation}
    \nonumber
	\mathfrak{X}(M)^{\mathbb{C}} = \{X + iY; \ X,Y \in \mathfrak{X}(M)\}.
\end{equation} 

For a smooth complex function $f:M^n\to \mathbb{C}$, that is, $f = f_1 + if_2$ with $f_1,f_2 \in C^{\infty}(M)$, the gradient of $f$ is the vector field $\nabla{f} \in \mathfrak{X}(M)^{\mathbb{C}}$ defined on $M$ by
\begin{equation}
	\nonumber
    \langle\!\langle\nabla{f}, \overline{Z}\rangle\!\rangle_p= Z(f) = df(Z),
\end{equation}
for all $Z\in \mathfrak{X}(M)^{\mathbb{C}}$ and we have
\begin{equation}
    \nonumber
	\nabla{f} = \nabla{f_{1}} + i\nabla{f_{2}}.
\end{equation}
Moreover, for $Z=Z_1+iZ_2$ the following identities are valid
\begin{equation}
\nonumber
\begin{split}
    \mathrm{div} Z &= \mathrm{div}Z_1 + i\, \mathrm{div}Z_2,\\
	\mathrm{div}(fZ) &= f \mathrm{div}Z + \langle\!\langle\nabla{f}, \overline{Z}\rangle\!\rangle_p
\end{split}
\end{equation}
and
\begin{equation} 
    \nonumber
	\Delta{f} = \mathrm{div}(\nabla{f}) = \Delta{f_1} + i\Delta{f_2}.
\end{equation}
		
Analogously, for a function $\eta \in C^{\infty}(M)$ we get the $\eta $-divergence as follows  
\begin{equation}
    \nonumber
	\mathrm{div}_{\eta}(T(fZ)) = f\, \mathrm{div}_{\eta}(T(Z)) + \langle\!\langle\nabla{f}, \overline{T(Z)}\rangle\!\rangle_p,
\end{equation}
where the tensor  $T$ satisfies $T(X + iY) = T(X) + iT(Y)$.
Thus, it is natural to extend the operator $\mathcal{L}$ for the complex case as follows
\begin{equation}\label{2.11}
	\mathcal{L}(f) = \mathrm{div}_{\eta}(T(\nabla{f})) = \mathrm{div}(T(\nabla{f})) - \langle\!\langle T(\nabla{f}), \nabla{\eta} \rangle\!\rangle_p.
\end{equation}
Therefore, for smooth complex functions $g,h$ we have
\begin{equation}
    \nonumber
	\mathcal{L}{h} = \mathcal{L}{h_1} + i\mathcal{L}{h_2}
\end{equation}
	and
\begin{equation}
    \nonumber
	\mathcal{L}{(hg)} = h\mathcal{L}{g} + g\mathcal{L}{h} + 2\langle\!\langle T(\nabla h), \overline{\nabla g} \rangle\!\rangle_p.
\end{equation}
Then, for the complex case, the divergence theorem takes the following form
\begin{equation}
\nonumber
\begin{split}
    \int_{\Omega}h\mathcal{L}g \mathrm{dm} =  -\int_{\Omega}\langle\!\langle T(\nabla{h}),\overline{\nabla{g}}\rangle\!\rangle_p\mathrm{dm} + \int_{\partial\Omega} h\langle\!\langle\nu,T(\overline{\nabla{g}})\rangle\!\rangle_p\mathrm{d\tau}.
\end{split}
\end{equation}

To finish this section we will highlight an important result of a Cartan-Hadmard manifold. 

In what follows, we consider $\Omega$ be a bounded domain of $M$ and let us denote by $r:\overline{\Omega} \rightarrow \mathbb{R}, r(x)=d(x,o)$ the distance function from some fixed point $o\in M\!\setminus\!\overline{\Omega}$, that is differentiable, with $|\nabla r|=1$ and $\nabla r=\partial_{r}$.

Regarding Cartan-Hadamard manifolds, we would like to recall the following result.
	
\begin{lem}[\textbf{Rauch Comparison, \cite{PP}, p. 255}]\label{RC}
	Assume that $(M^n,\langle\cdot,\cdot\rangle)$ satisfies $c\leq{K}\leq{C}$. If $\langle\cdot,\cdot\rangle={dr}^2+g_r$ represents the metric in the polar coordinates, then
    \begin{equation}
        \nonumber
		\frac{sn'_{C}(r)}{sn_{C}(r)}g_r\leq\nabla^{2}r\leq\frac{sn'_{c}(r)}{sn_{c}(r)}g_r,
	\end{equation}
    where $sn_{k}$ denotes the unique solution to \"x$(r)+k\cdot{x(r)}=0$ with x$(0)=0$ and \.x$(0)=1$.
\end{lem}

In particular, $\frac{sn'_{k}(r)}{sn_{k}(r)}=\sqrt{-k}\frac{\cosh(\sqrt{-k}r)}{\sinh(\sqrt{-k}r)}$ when $k<0$, $\frac{sn'_{k}(r)}{sn_{k}(r)}=\frac{1}{r}$ when $k=0$, and $\frac{sn'_{k}(r)}{sn_{k}(r)}=\sqrt{k}\frac{\cos(\sqrt{k}r)}{\sin(\sqrt{k}r)}$ when $k>0$.
    
\section{Auxiliary Results}
	
We begin this section with first auxiliary Lemma due to Chen, Zheng and Yang, in \cite{Art}.
		
\begin{lem}[Chen, Zheng and Yang, see \cite{Art}, p. 298]\label{L3.1}
	Assume that ${\{\mu_j\}}_{j=1}^{\infty}$ is a nondecreasing sequence, that is,
	\begin{equation}
        \nonumber
        0<\mu_1\leq\mu_2\leq\cdots\leq \mu_k\leq\cdots \to + \infty,
	\end{equation}
	where each $\mu_j$ has finite multiplicity $m_j$ and is repeated according to its multiplicity. 
	Let $r={(r_j)}_{j=1}^{\infty}\in \ell_2$ such that $r_{m_1}\neq 0$ and $\displaystyle \sum_{j=1}^{\infty}\mu_jr_j^{2} < \sqrt{AB}$ with
	\begin{equation}
        \nonumber
		B =\displaystyle \sum_{j=1}^{\infty} r_j^{2} \quad \mbox{and} \quad A = \displaystyle \sum_{j=1}^{\infty} \mu_j^{2}r_j^{2}<+\infty\ ,
	\end{equation} 
	then
	\begin{equation}
        \nonumber
		\sum_{j=1}^{\infty}\mu_jr_j^{2} \leq 
		\dfrac{A + \mu_{m_1}\mu_{{m_1}+1}B}{\mu_{m_1} + \mu_{{m_1}+1}}. 
	\end{equation}
\end{lem}
A detailed proof of the Lemma~\ref{L3.1} can be found in \cite{CSS}.
			
For the next result, note that if $g$ is a complex smooth function in $M^n$ we have
\begin{equation}
    \nonumber
	\mathcal{L}(|g|^{2}) = 2g_1\mathcal{L}{g_1} + 2g_2\mathcal{L}{g_2} + 2\langle\!\langle T(\nabla g), \nabla g\rangle\!\rangle_p,
\end{equation}
then we can conclude $\langle\!\langle T(\nabla g),\nabla g\rangle\!\rangle_p$ is a real number for all $p\in M^n$.
	
\begin{lem}\label{L3.2}
	For the Dirichlet eigenvalue problem \eqref{problem1}, let ${\{u_k\}}_{k=1}^{\infty}$ be the orthonormal
    eigenfunction corresponding to the $k$-th eigenvalue $\lambda_{k}$. Then for any complex-valued function $g\in C^3(\Omega)\cap C^{2}(\overline{\Omega})$ such that $gu_j$ is not the $\mathbb{C}$-linear combination of $u_1, u_2, \cdots, u_{k+1}$ and such that
	\begin{equation}
        \nonumber
		\int_{\Omega}gu_{j}u_{k+1} \mathrm{dm} \neq 0,
	\end{equation}
	for $\lambda_{j}<\lambda_{k+1}<\lambda_{k+2}$, $\quad k,j\in \mathbb{Z}^+$, $\quad j\geq 1$, we have
	\begin{equation}\label{3.1}
    \begin{split}
		\big[(\lambda_{k+1} &- \lambda_{j}) + (\lambda_{k+2} - \lambda_{j})\big]\int_{\Omega}\langle\!\langle \nabla{g},T(\nabla{g})\rangle\!\rangle_p u_{j}^{2} \mathrm{dm}\\
		&\leq \int_{\Omega}
		\Big|2\langle\!\langle T(\nabla{u_{j}}), \nabla{\overline{g}}\rangle\!\rangle_p + u_{j}\mathcal{L}{g}\Big|^{2} \mathrm{dm} + (\lambda_{k+2} - \lambda_{j})(\lambda_{k+1} - \lambda_{j})\int_{\Omega}|gu_{j}|^{2} \mathrm{dm}.
    \end{split}
	\end{equation}
\end{lem}
\begin{proof}
	We will use the Lemma~\ref{L3.1}. Start by defining
	\begin{equation}
    \nonumber
		a_{js} = \int_{\Omega}gu_{j}u_{s} \mathrm{dm} \quad \mbox{and} \quad		b_{js} = \int_{\Omega} \left(u_{j}\mathcal{L}{g} + 2\langle\!\langle T(\nabla{u_{j}}), \nabla{\overline{g}}\rangle\!\rangle_p \right)\!u_{s} \mathrm{dm}.
	\end{equation}
	Note that $a_{js} = a_{sj}$, and using integration by parts formula and \eqref{2.11} we obtain
    \begin{equation}
    \nonumber
        \begin{split}
		\lambda_{s}a_{js} &= - \int_{\Omega}gu_{j}(\mathcal{L}{u_{s}})\mathrm{dm}
		= -\int_{\Omega}\left(\mathcal{L}{gu_{j}}\right)u_{s} \mathrm{dm} \\ &= -\int_{\Omega}\Big(g\mathcal{L}{u_{j}} + u_{j}\mathcal{L}{g} + 2\langle\!\langle T(\nabla{u_{j}}), \nabla{\overline{g}}\rangle\!\rangle_p \Big)u_{s}\mathrm{dm} \\
		&= -\int_{\Omega}\Big(-	g\lambda_{j} u_{j} + u_{j}\mathcal{L}{g}  + 2\langle\!\langle T(\nabla{u_{j}}), \nabla{\overline{g}}\rangle\!\rangle_p \Big) u_{s}\mathrm{dm} \\
		&= \lambda_{j}\int_{\Omega} 
		gu_{j}u_{s}\mathrm{dm} -\int_{\Omega}\Big(u_{j}\mathcal{L}{g} + 2\langle\!\langle T(\nabla{u_{j}}), \nabla{\overline{g}}\rangle\!\rangle_p \Big) u_{s}\mathrm{dm} \\
		&= \lambda_{j}a_{js} - b_{js},
        \end{split}
	\end{equation}
	that is,
	\begin{equation}\label{3.2}
		\lambda_{s}a_{js} - \lambda_{j}a_{js} = 
		(\lambda_{s} - \lambda_{j})a_{js} = -b_{js}.
	\end{equation}
	Since $\{u_k\}_{k=1}^{\infty}$ is complete and orthonormal basis on $L_2(\Omega,\mathrm{dm})$, using the inner product of $L_2(\Omega,\mathrm{dm})$, Parseval identity and definition of the $a_{js}$, we have
	\begin{equation}\label{3.3}
    \begin{split}
		\int_{\Omega}|gu_{j}|^{2} \mathrm{dm} &= \big{|}gu_{j}\big{|}^{2}_{L_2(\Omega,\mathrm{dm})} = \sum_{s=1}^{\infty}\big|\langle gu_{j}, u_{s} \rangle\big|^{2} \\
		&= \sum_{s=1}^{\infty}\left| \int_{\Omega}gu_{j}\overline{u_{s}} \mathrm{dm} \right|^{2}  = \sum_{s=1}^{\infty}\left| \int_{\Omega}gu_{j} u_{s} \mathrm{dm} \right|^{2}  = \sum_{s=1}^{\infty}| a_{js}|^{2}.
    \end{split}
	\end{equation}
	Whereas $u_{j}$ is a real function for all $j$ and using Stokes' theorem, we get
	\begin{equation}\label{3.4}
	\begin{split}
		\int_{\Omega}\langle\!\langle\nabla{g}, T(\nabla{g})\rangle\!\rangle_p u_{j}^{2} \mathrm{dm}
		&= \int_{\Omega}\langle\!\langle\nabla{g},T(\overline{u}_{j}^{2}\nabla{g})\rangle\!\rangle_p \mathrm{dm} = - \int_{\Omega}g \left(\mathrm{div}_{\eta}(T(u_{j}^{2}\nabla\overline{g}))\right) \mathrm{dm} \\
		&= - \int_{\Omega}g\Big(u_{j}^{2}\mathrm{div}_{\eta}(T(\nabla\overline{g})) + \langle\!\langle T(\nabla{u_{j}^{2}}), \overline{\nabla\overline{g}}\rangle\!\rangle_p\Big) \mathrm{dm} \\
		&= - \int_{\Omega}g\Big(u_{j}^{2} \mathcal{L}\overline{g} + 2u_{j}\langle\!\langle T(\nabla{u_{j}}), \nabla{g}\rangle\!\rangle_p\Big) \mathrm{dm} \\
		&= - \int_{\Omega} gu_{j}\left( u_{j}\mathcal{L}\overline{g} + 2\langle\!\langle T(\nabla{u_{j}}), \nabla{g}\rangle\!\rangle_p \right) \mathrm{dm}.
		\end{split}
	\end{equation}
	From the definitions of $a_{js}$ and $b_{js}$, and identities \eqref{3.2} and \eqref{3.4} we have
	\begin{equation}\label{3.5}
	\begin{split}
		\int_{\Omega}\langle\!\langle \nabla{g},T(\nabla{g})\rangle\!\rangle_p u_{j}^{2} \mathrm{dm}
		&= - \int_{\Omega} gu_{j}\left( u_{j}\mathcal{L}\overline{g} + 2\langle\!\langle T(\nabla{u_{j}}), \nabla{g}\rangle\!\rangle_p \right)\mathrm{dm} \\
		&= - \int_{\Omega}gu_{j}\left( \overline{u_{j}\mathcal{L}{g} + 2\langle\!\langle T(\nabla{u_{j}}), \nabla{\overline{g}}\rangle\!\rangle_p}
		\right) \mathrm{dm} \\
		&= - \left\langle gu_{j},u_{j}\mathcal{L}{g} + 2\langle\!\langle T(\nabla{u_{j}}), \nabla{\overline{g}}\rangle\!\rangle_p
		\right\rangle \\
		&= - \left\langle \sum_{s=1}^{\infty}\left\langle gu_{j}, u_{s} \right\rangle u_{s} ,\sum_{k=1}^{\infty} \left\langle u_{j}\mathcal{L}{g} + 2\langle\!\langle T(\nabla{u_{j}}), \nabla{\overline{g}}\rangle\!\rangle_p, u_k \right\rangle u_k \right\rangle \\
		&= - \sum_{s=1}^{\infty}\left\langle gu_{j}, u_{s}\right\rangle\overline{\sum_{k=1}^{\infty} \left\langle u_{j}\mathcal{L}{g} + 2\langle\!\langle T(\nabla{u_{j}}), \nabla{\overline{g}}\rangle\!\rangle_p, u_k\right\rangle}\left\langle u_{s} ,u_k\right\rangle \\
		&= - \sum_{s,k=1}^{\infty}\left\langle gu_{j}, u_{s}\right\rangle\overline{\left\langle u_{j}\mathcal{L}{g} + 2\langle\!\langle T(\nabla{u_{j}}), \nabla{\overline{g}}\rangle\!\rangle_p, u_k\right\rangle}\delta_{sk} \\
		&= - \sum_{s=1}^{\infty}\left\langle gu_{j}, u_{s}\right\rangle\overline{\left\langle u_{j}\mathcal{L}{g} + 2\langle\!\langle T(\nabla{u_{j}}), \nabla{\overline{g}}\rangle\!\rangle_p, u_{s}\right\rangle} \\
		&= - \sum_{s=1}^{\infty}\int_{\Omega}gu_{j}u_{s} \mathrm{dm} \overline{\int_{\Omega}\left( 
		u_{j}\mathcal{L}{g} + 2\langle\!\langle T(\nabla{u_{j}}), \nabla{\overline{g}}\rangle\!\rangle_p\right)\! u_{s} \mathrm{dm}} \\
		&= - \sum_{s=1}^{\infty}a_{js} \overline{b_{js}} 
		= \sum_{s=1}^{\infty} a_{js}\overline{(-b_{js})} 
		= \sum_{s=1}^{\infty} a_{js}\overline{(\lambda_{s} - \lambda_{j})a_{js}} \\
		&= \sum_{s=1}^{\infty}(\lambda_{s} - \lambda_{j}) a_{js}\overline{a_{js}}
		= \sum_{s=1}^{\infty}(\lambda_{s} - \lambda_{j}) |a_{js}|^{2}.
	\end{split}
	\end{equation}
	Again from Parseval identity and from equation~\eqref{3.2}, we get
		%\allowdisplaybreaks
	\begin{equation*}
	\begin{split}
		\int_{\Omega}\big|u_{j}\mathcal{L}{g} + 2\langle\!\langle T(\nabla{u_{j}}), \nabla{\overline{g}}\rangle\!\rangle_p\big|^{2} \mathrm{dm}
		&= \Big{|} u_{j}\mathcal{L}{g} + 2\langle\!\langle T(\nabla{u_{j}}), \nabla{\overline{g}}\rangle\!\rangle_p \Big{|}_{L_2(\Omega,\mathrm{dm})}^{2} \\
		&= \sum_{s=1}^{\infty}\left|\Big\langle u_{j}\mathcal{L}{g} + 2\langle\!\langle T(\nabla{u_{j}}), \nabla{\overline{g}}\rangle\!\rangle_p, u_{s}\Big\rangle \right|^{2} \\
		&= \sum_{s=1}^{\infty} \left|\int_{\Omega} \Big(u_{j}\mathcal{L}{g} + 2\langle\!\langle T(\nabla{u_{j}}), \nabla{\overline{g}}\rangle\!\rangle_p\Big) \overline{u_{s}}  \mathrm{dm} \right|^{2}, \\ 
	\end{split}
	\end{equation*}
that is,
 \begin{equation}\label{3.6}
	\begin{split}
		\int_{\Omega}\big|u_{j}\mathcal{L}{g} + 2\langle\!\langle T(\nabla{u_{j}}), \nabla{\overline{g}}\rangle\!\rangle_p\big|^{2} \mathrm{dm}
        &= \sum_{s=1}^{\infty} \left|\int_{\Omega} \left(u_{j}\mathcal{L}{g} + 2\langle\!\langle T(\nabla{u_{j}}), \nabla{\overline{g}}\rangle\!\rangle_p\right) u_{s} \mathrm{dm}\right|^{2}\\
		&= \sum_{s=1}^{\infty}\left| b_{js} \right|^{2} = \sum_{s=1}^{\infty} \left|(\lambda_{j} - \lambda_{s})a_{js}\right|^{2} \\
		&= \sum_{s=1}^{\infty} \big|(\lambda_{j} - \lambda_{s})\big|^{2}\big|a_{js}\big|^{2}
		= \sum_{s=1}^{\infty} (\lambda_{s} - \lambda_{j})^{2}\left|a_{js}\right|^{2}.
	\end{split}
	\end{equation}
	From equality~\eqref{3.5}, we deduce
	\begin{equation}
    \nonumber
    \begin{split}
		\Bigg(\int_{\Omega}\langle\!\langle &\nabla{g},T(\nabla{g})\rangle\!\rangle_p u_{j}^{2} \mathrm{dm} - \sum_{s=1}^{k}(\lambda_{s} - \lambda_{j})|a_{js}|^{2} \Bigg)^{2} \\
		&= \left(\sum_{s=1}^{\infty}
		(\lambda_{s} - \lambda_{j})|a_{js}|^{2} -\sum_{s=1}^{k}(\lambda_{s} - \lambda_{j})|a_{js}|^{2}\right)^{2} = \left(\sum_{s=k+1}^{\infty}
		(\lambda_{s} - \lambda_{j})|a_{js}|^{2}\right)^{2}.
    \end{split}
	\end{equation}
	Follow from identities \eqref{3.3}, \eqref{3.6} and Cauchy-Schwarz inequality, that 
	{\small\begin{equation}
    \nonumber
    \begin{split}
       &\Bigg(\sum_{s=k+1}^{\infty}
		(\lambda_{s} - \lambda_{j})
		|a_{js}|^{2}\Bigg)^{2} \\
		  &= \left(\sum_{s=k+1}^{\infty}
		\Big((\lambda_{s} - \lambda_{j})
		|a_{js}|\Big)|a_{js}|\right)^{2} \\
		&\leq \left(\sum_{s=k+1}^{\infty}
		\big(|a_{js}|\big)^{2}\right)  \negthickspace \left(\sum_{s=k+1}^{\infty}
		\Big((\lambda_{s} - \lambda_{j})
		|a_{js}|\Big)^{2}\right) \\
		& = \left(\sum_{s=1}^{\infty}
		|a_{js}|^{2} - \sum_{s=1}^{k}
		|a_{js}|^{2}\right) \negthickspace
		\left(\sum_{s=1}^{\infty}
		(\lambda_{s} - \lambda_{j})^{2}|a_{js}|^{2} - \sum_{j=1}^{k}(\lambda_{s} - \lambda_{j})^{2}
		|a_{js}|^{2} \right)\\
		&= \left(\int_{\Omega}|gu_{j}|^{2} \mathrm{dm} - \sum_{s=1}^{k}
		|a_{js}|^{2}\right) \negthickspace
		\left(\int_{\Omega}\Big|u_{j}\mathcal{L}{g} + 2\langle\!\langle T(\nabla{u_{j}}), \nabla{\overline{g}}\rangle\!\rangle_p \Big|^{2} \mathrm{dm} - \sum_{s=1}^{k}
		(\lambda_{s} - \lambda_{j})^{2}
		|a_{js}|^{2}\right),
    \end{split}
	\end{equation}}
	that is,
	{\small\begin{equation}\label{3.7}
	\begin{split}
		&\left(\int_{\Omega}\langle\!\langle \nabla{g},T(\nabla{g})\rangle\!\rangle_p u_{j}^{2}\mathrm{dm} - \sum_{s=1}^{k}(\lambda_{s} - \lambda_{j})|a_{js}|^{2}\right)^{2} \\ &\leq \left(\int_{\Omega}|gu_{j}|^{2} \mathrm{dm} - \sum_{s=1}^{k}
		|a_{js}|^{2}\right) \negthickspace
		\left(\int_{\Omega}\Big|u_{j}\mathcal{L}{g} + 2\langle\!\langle T(\nabla{u_{j}}), \nabla{\overline{g}}\rangle\!\rangle_p \Big|^{2} \mathrm{dm} - \sum_{s=1}^{k}
		(\lambda_{s} - \lambda_{j})^{2}
		|a_{js}|^{2}\right).
	\end{split}
	\end{equation}}
		
	Now, as $a_{j{k+1}}\displaystyle\int_{\Omega} gu_{j}u_{k+1}\neq 0$, define 		
	\begin{equation}
    \nonumber
    \begin{split}
		{A}(j) &= \int_{\Omega}\Big|u_{j}\mathcal{L}{g} + 2\langle\!\langle T(\nabla{u_{j}}), \nabla{\overline{g}}\rangle\!\rangle_p \Big|^{2} \mathrm{dm} - \sum_{s=1}^{k}
		(\lambda_{s} - \lambda_{j})^{2}|a_{js}|^{2}\\ 
		&= \sum_{s=k+1}^{\infty}
		(\lambda_{s} - \lambda_{j})^{2}|a_{js}|^{2} > 0, \\[8pt]
		{B}(j) &= \int_{\Omega}|gu_{j}|^{2} \mathrm{dm} - \sum_{s=1}^{k}|a_{js}|^{2} = \sum_{s=k+1}^{\infty}|a_{js}|^{2} > 0, \\[8pt]
		{C}(j) &= \int_{\Omega}
		\langle\!\langle \nabla{g},T(\nabla{g})\rangle\!\rangle_p u_{j}^{2} \mathrm{dm} - \sum_{s=1}^{k}(\lambda_{s} - \lambda_{j})
		|a_{js}|^{2} = \sum_{s=k+1}^{\infty}(\lambda_{s} - \lambda_{j})
		|a_{js}|^{2} > 0.
    \end{split}
	\end{equation}
	Notice that $gu_{j}$ is not the $\mathbb{C} $-linear combination of $u_1,\cdots,u_{k+1}$ then there are some $l>k+1$ such that
	\begin{equation}
        \nonumber
		a_{jl} = \int_{\Omega}gu_{j}u_l \neq 0.
	\end{equation}
	Moreover $\lambda_{j} <\lambda_{k+1} < \lambda_{k+2} \leq \lambda_l$, the vector ${(|a_{js}|)}_{s=k+1}^{\infty}$ is not proportional to ${\left((\lambda_{s}-\lambda_{j})^{2}|a_{js}|\right)}_{s=k+1}^{\infty}$, and from Cauchy-Schwarz inequality, we conclude that \eqref{3.7} is equivalent to
	\begin{equation}\label{3.8}
		{C}(j) < \sqrt{{A}(j){B}(j)}.
	\end{equation}
	Since $a_{j{k+1}}\neq0$ and from \eqref{3.8}, applying the Theorem~\ref{L3.1} we deduce
	\begin{equation}\label{3.9}
		{C}(j) \leq \dfrac{{A}(j) + (\lambda_{k+2} - \lambda_{j})(\lambda_{k+1} - \lambda_{j}){B}(j)}{(\lambda_{k+2} - \lambda_{j}) + (\lambda_{k+1} - \lambda_{j})}.
	\end{equation}
	From \eqref{3.9} and the definitions of ${A}(j),{B}(j)$ and ${C}(j)$, we obtain
	\begin{equation}
    \nonumber
    \begin{split}
		\big((\lambda_{k+2} - \lambda_{j}) +& (\lambda_{k+1} - \lambda_{j})\big)\left(\int_{\Omega}
		\langle\!\langle \nabla{g},T(\nabla{g})\rangle\!\rangle_p u_{j}^{2} \mathrm{dm} - \sum_{s=1}^{k}(\lambda_{s} - \lambda_{j})|a_{js}|^{2}\right) \\
		&\leq \int_{\Omega}\Big|u_{j}\mathcal{L}{g} + 2\langle\!\langle T(\nabla{u_{j}}), \nabla{\overline{g}}\rangle\!\rangle_p \Big|^{2}  \mathrm{dm} - \sum_{s=1}^{k}(\lambda_{s} - \lambda_{j})^{2}|a_{js}|^{2} \\
		&\quad + (\lambda_{k+2} - \lambda_{j})(\lambda_{k+1}  - \lambda_{j})\left(\int_{\Omega}|gu_{j}|^{2} \mathrm{dm} - \sum_{s=1}^{k}|a_{js}|^{2}\right),
    \end{split}
	\end{equation}
	which implies
    \begin{equation}
    \nonumber
    \begin{split}
		&\left((\lambda_{k+2} - \lambda_{j}) + (\lambda_{k+1} - \lambda_{j})\right)\int_{\Omega}
		\langle\!\langle \nabla{g},T(\nabla{g})\rangle\!\rangle_p u_{j}^{2} \mathrm{dm} \\
		&\quad \quad \leq \int_{\Omega}
		\Big|u_{j}\mathcal{L}{g} + 2\langle\!\langle T(\nabla{u_{j}}), \nabla{\overline{g}}\rangle\!\rangle_p\Big|^{2} \mathrm{dm} + (\lambda_{k+2} - \lambda_{j})(\lambda_{k+1} - \lambda_{j})\int_{\Omega}|gu_{j}|^{2} \mathrm{dm} \\
		&\quad \quad \quad + \left[\Big((\lambda_{k+2} - \lambda_{j}) + (\lambda_{k+1} - \lambda_{j})\Big)
		\sum_{s=1}^{k}(\lambda_{s} - \lambda_{j})|a_{js}|^{2} - \sum_{s=1}^{k}(\lambda_{s} - \lambda_{j})^{2}|a_{js}|^{2}\right] \\
		&\quad \quad \quad - (\lambda_{k+2} - \lambda_{j})(\lambda_{k+1} - \lambda_{j})\sum_{s=1}^{k}|a_{js}|^{2} \\[12pt]
		&\quad \quad = \int_{\Omega}
		\Big|u_{j}\mathcal{L}{g} + 2\langle\!\langle T(\nabla{u_{j}}), \nabla{\overline{g}}\rangle\!\rangle_p\Big|^{2}\! \mathrm{dm} + (\lambda_{k+2} - \lambda_{j})(\lambda_{k+1} - \lambda_{j})\int_{\Omega}|gu_{j}|^{2} \mathrm{dm} \\
		&\quad \quad \quad + \sum_{s=1}^{k}\Big[(\lambda_{k+2} - \lambda_{j})(\lambda_{s} - \lambda_{j}) + (\lambda_{k+1} - \lambda_{j})
		(\lambda_{s} - \lambda_{j}) \\
		&\quad \quad \quad - \quad (\lambda_{k+2} - \lambda_{j})(\lambda_{k+1} - \lambda_{j}) - (\lambda_{s} - \lambda_{j})^{2}\Big]|a_{js}|^{2} \\[12pt]
		&\quad \quad = \int_{\Omega}
		\Big|u_{j}\mathcal{L}{g} + 2\langle\!\langle T(\nabla{u_{j}}), \nabla{\overline{g}}\rangle\!\rangle_p\Big|^{2}\! \mathrm{dm} + (\lambda_{k+2} - \lambda_{j})(\lambda_{k+1} - \lambda_{j})\int_{\Omega}|gu_{j}|^{2} \mathrm{dm} \\
		&\quad \quad \quad - \sum_{s=1}^{k}\Big[(\lambda_{k+2} - \lambda_{s})(\lambda_{k+1} - \lambda_{s})\Big]|a_{js}|^{2},
    \end{split}
	\end{equation}
    so
	\begin{equation}
    \nonumber
    \begin{split}
		&\left((\lambda_{k+2} - \lambda_{j}) + (\lambda_{k+1} - \lambda_{j})\right)\int_{\Omega}
		\langle\!\langle \nabla{g},T(\nabla{g})\rangle\!\rangle_p u_{j}^{2} \mathrm{dm} \\
		&\quad \quad \leq \int_{\Omega}
		\Big|u_{j}\mathcal{L}{g} + 2\langle\!\langle T(\nabla{u_{j}}), \nabla{\overline{g}}\rangle\!\rangle_p\Big|^{2}\! \mathrm{dm} + (\lambda_{k+2} - \lambda_{j})(\lambda_{k+1} - \lambda_{j})\int_{\Omega}|gu_{j}|^{2} \mathrm{dm}.
    \end{split}
	\end{equation}
    this concludes the proof.
\end{proof}

Then, by choosing the appropriate function in the Lemma above we obtain the following result.
    
\begin{cor}\label{C3.1}
    Under the assumption of Lemma~\ref{L3.2}, for any nonconstant real-valued function $f\in C^3(\Omega)\cap C^{2}(\overline{\Omega})$, we have
    {\small\begin{equation}
	\nonumber
    \begin{split}
		\Big((&\lambda_{k+2} - \lambda_j) + (\lambda_{k+1} -\lambda_j)\Big)\!
		\int_{\Omega}\langle \nabla{f},T(\nabla{f})\rangle u_j^{2} \mathrm{dm} \\
		& \leq 2\sqrt{(\lambda_{k+2} \!- \!\lambda_j)(\lambda_{k+1} \!- \!\lambda_j)\!\int_{\Omega}\!\langle\nabla{f},T(\nabla{f})\rangle^{2} u_j^{2} \mathrm{dm} }
		+\!\int_{\Omega}\!\Big(2\langle T(\nabla{u_j}), \nabla{f} \rangle \!+\! u_j\mathcal{L}{f}\Big)^{2} \!\mathrm{dm}.
    \end{split}
	\end{equation}}
\end{cor}
	
\begin{proof}
    We will start applying the Lemma~\ref{L3.2} for $g={e^{({i}\alpha f)}}$, where $f$ is a non-constant function $f\in C^3(\Omega)\cap C^{2}(\overline{\Omega})$, $\alpha\in\mathbb{R}\setminus\{0\}$ and $i=\sqrt{-1}$. We have
	\begin{equation}
        \nonumber
		\nabla{g} = \nabla({e^{{i}\alpha f}}) 
		= {e^{{i}\alpha f}} ({i}\alpha \nabla{f}),
	\end{equation}
    then
	\begin{equation}
    \nonumber
    \begin{split}
		\mathcal{L}{g}:&= \mathrm{div}_{\eta}(T(\nabla{g}) = \mathrm{div}(T(\nabla{g})) - \langle\!\langle T(\nabla{g}),\nabla{\eta}\rangle\!\rangle_p\\
		&= \mathrm{div}({e^{{i}\alpha f}}{i}\alpha T(\nabla{f})) - \langle\!\langle {i}\alpha {e^{{i}\alpha f}} T(\nabla{f}) , \nabla{\eta} \rangle\!\rangle_p\\
		&= {i}\alpha {e^{{i}\alpha f}} \mathrm{div}(T(\nabla{f})) + \langle\!\langle T(\nabla{f}) , \overline{-{\alpha}^{2} {e^{{i}\alpha f}}\nabla{f}} \rangle\!\rangle_p- {i}\alpha {e^{{i}\alpha f}}\langle\!\langle T(\nabla{f}) , \nabla{\eta} \rangle\!\rangle_p\\
		&=  {i}\alpha {e^{{i}\alpha f}}\mathrm{div}_{\eta}(T(\nabla{f})) - {\alpha}^{2} {e^{{i}\alpha f}} \langle\!\langle T(\nabla{f}) , \nabla{f} \rangle\!\rangle_p\\
		&= {i}\alpha ({e^{{i}\alpha f}})\mathcal{L}{f} - {\alpha}^{2} {e^{{i}\alpha f}} \langle\!\langle T(\nabla{f}) , \nabla{f} \rangle\!\rangle_p
    \end{split}
	\end{equation}
    and
	\begin{equation}\label{3.10}
		\int_{\Omega}|gu_j|^{2} \mathrm{dm} = \int_{\Omega}\big|{e^{{i}\alpha f}}u_j\big|^{2} \mathrm{dm}
		= \int_{\Omega}\big|{e^{{i}\alpha f}}\big|^{2} |u_j|^{2} \mathrm{dm} = \int_{\Omega}u_j^{2} \mathrm{dm} = 1.
	\end{equation}
    Therefore we obtain
	\begin{equation}\label{3.11}
	\begin{split}
		\int_{\Omega}\langle\!\langle \nabla{g}, T(\nabla{g}) \rangle\!\rangle_pu_j^{2}\ \mathrm{dm} &= \int_{\Omega}{i}\alpha
		{e^{{i}\alpha f}}\langle\!\langle \nabla{f} , {i}\alpha{e^{{i}\alpha f}}T(\nabla{f}) \rangle\!\rangle_p\mathrm{dm} \\
		&= \int_{\Omega}{i}\alpha
		{e^{{i}\alpha f}}\overline{{i}\alpha{e^{{i}\alpha f}}}\langle\!\langle \nabla{f}, T(\nabla{f}) \rangle\!\rangle_p\mathrm{dm} \\
		&=\int_{\Omega}|{i}\alpha ({e^{{i}\alpha f}})|^{2}\langle \nabla{f}, T(\nabla{f}) \rangle u_j^{2}\ \mathrm{dm} \\
		&= \alpha^{2}\int_{\Omega}
		\langle \nabla{f}, T(\nabla{f}) \rangle u_j^{2}\ \mathrm{dm}
		\end{split}
	\end{equation} 
    and
    \begin{equation}
	\begin{split}
		\int_{\Omega}&\Big|2\langle\!\langle T(\nabla{u_j}),\nabla\overline{g} \rangle\!\rangle_p+ u_j\mathcal{L}{g}\Big|^{2} \mathrm{dm} \\
        &= \int_{\Omega}\Big|2\langle\!\langle T(\nabla{u_j}),\overline{{i}\alpha{e^{{i}\alpha f}}\nabla{f}} \rangle\!\rangle_p+ u_j\mathcal{L}{g}\Big|^{2} \mathrm{dm} \\
		&= \int_{\Omega}\Big|2{i}\alpha {e^{{i}\alpha f}}\langle\!\langle T(\nabla{u_j}),\nabla{f} \rangle\!\rangle_p + u_j{i}\alpha 
		{e^{{i}\alpha f}}\mathcal{L}{f} - u_j\alpha^{2}
		{e^{{i}\alpha f}}\langle\!\langle T(\nabla{f}),\nabla{f} \rangle\!\rangle_p\Big|^{2} \mathrm{dm} \\[12pt]
		&= \int_{\Omega}\Big|{e^{{i}\alpha f}}\Big[{i}\alpha \big(2 \langle T(\nabla{u_j}), \nabla{f} \rangle + u_j\mathcal{L}{f}\big) - u_j\alpha^{2}\langle T(\nabla{f}),\nabla{f} \rangle \Big]\Big|^{2} \mathrm{dm} \\[12pt]
		&= \int_{\Omega}\Big|{i}\alpha \big(2 \langle T(\nabla{u_j}), \nabla{f} \rangle + u_j\mathcal{L}{f}\big) - u_j\alpha^{2}\langle T(\nabla{f}),\nabla{f} \rangle \Big|^{2} \mathrm{dm},\nonumber
	\end{split}
	\end{equation}
    that is,
    \begin{equation}\label{3.12}
	\begin{split}
		\int_{\Omega}&\Big|2\langle\!\langle T(\nabla{u_j}),\nabla\overline{g} \rangle\!\rangle_p+ u_j\mathcal{L}{g}\Big|^{2} \mathrm{dm} \\
        &= \alpha^{2}\int_{\Omega} \Big(2\langle T(\nabla{u_j}), \nabla{f} \rangle + u_j\mathcal{L}{f}\Big)^{2} \mathrm{dm} + \alpha^4\int_{\Omega} \langle \nabla{f},T(\nabla{f})\rangle^{2} u_j^{2}\ \mathrm{dm}.
	\end{split}
	\end{equation}
 
    Substituting \eqref{3.10}, \eqref{3.11} and \eqref{3.12} into \eqref{3.1} we have
	\begin{equation}
	\nonumber
    \begin{split}
		&\alpha^{2}\Big((\lambda_{k+2} - \lambda_j) +
		(\lambda_{k+1} - \lambda_j)\Big)\int_{\Omega}
		\langle \nabla{f}, T(\nabla{f}) \rangle u_j^{2}\ \mathrm{dm} \\
		&\quad \leq \alpha^4\int_{\Omega}
		\langle \nabla{f}, T(\nabla{f}) \rangle^{2} u_j^{2}\ \mathrm{dm}
		+ \alpha^{2}\int_{\Omega}\Big(2\langle T(\nabla{u_j}), \nabla{f} \rangle + u_j\mathcal{L}{f}\Big)^{2} \mathrm{dm} \\
		&\qquad + (\lambda_{k+2} - \lambda_j)(\lambda_{k+1} - \lambda_j),
    \end{split}
	\end{equation} 
    and dividing the previous expression by $\alpha^{2}$ on both sides, we get
	{\small\begin{equation}\label{3.13}
	\begin{split}
		&\Big((\lambda_{k+2} - \lambda_j) +
		(\lambda_{k+1} - \lambda_j)\Big)\!\int_{\Omega}
		\langle \nabla{f}, T(\nabla{f}) \rangle u_j^{2}\ \mathrm{dm} \\
		&\leq \alpha^{2}\!\int_{\Omega}\!
		\langle\nabla{f}, T(\nabla{f}) \rangle^{2} u_j^{2}\ \mathrm{dm}\!+\!\frac{1}{\alpha^{2}}
		(\lambda_{k+2}\!-\!\lambda_j)
		(\lambda_{k+1}\!-\!\lambda_j)
		\!+\!\int_{\Omega}\! \Big(2\langle T(\nabla{u_j}), \! \nabla{f} \rangle\!+\! u_j\mathcal{L}{f}\Big)^{2}\! \mathrm{dm}.
	\end{split}
	\end{equation}}
 
	Since the inequality \eqref{3.13} is valid for any $\alpha\neq 0$, $(\lambda_{k+2} - \lambda_j)(\lambda_{k+1} - \lambda_j) \neq 0$ and $\displaystyle\int_{\Omega}\langle \nabla{f}, T(\nabla{f}) \rangle^{2} u_j^{2}\ \mathrm{dm} \neq 0$, we can choose 
    \begin{equation}
        \nonumber
        \alpha^{2} = \left(\dfrac{(\lambda_{k+2} - \lambda_j)(\lambda_{k+1} - \lambda_j)}{\int_{\Omega}\langle \nabla{f}, T(\nabla{f}) \rangle^{2} u_j^{2}\ \mathrm{dm}} \right)^{\!\!\!\frac{1}{2}}
    \end{equation}
    to obtain
	{\small\begin{equation}
    \nonumber
    \begin{split}
		\Big(&(\lambda_{k+2} - \lambda_j) +
		(\lambda_{k+1} - \lambda_j)\Big)\int_{\Omega}
		\langle \nabla{f}, T(\nabla{f}) \rangle u_j^{2}\ \mathrm{dm} \\ 
		\leq& \left(\dfrac{(\lambda_{k+2} - \lambda_j)(\lambda_{k+1} - \lambda_j)
		}{\int_{\Omega}	\langle \nabla{f}, T(\nabla{f}) \rangle^{2} u_j^{2}\ \mathrm{dm}} \right)^{\frac{1}{2}}
		\int_{\Omega}\langle \nabla{f}, T(\nabla{f}) \rangle^{2} u_j^{2}\ \mathrm{dm}\!+\! \int_{\Omega} \!\Big(2\langle T(\nabla{u_j}), \nabla{f} \rangle \!+\! u_j\mathcal{L}{f}\Big)^{2} \! \mathrm{dm} \\
		&+ \left(\dfrac{\int_{\Omega}
		\langle \nabla{f}, T(\nabla{f}) \rangle^{2} u_j^{2}\ \mathrm{dm}}{(\lambda_{k+2} - \lambda_j)(\lambda_{k+1} - \lambda_j)
		}\right)^{\frac{1}{2}}(\lambda_{k+2} - \lambda_j)(\lambda_{k+1} - \lambda_j) \\[10pt]
		=& \Big((\lambda_{k+2} \!-\! \lambda_j)(\lambda_{k+1} \!-\! \lambda_j)
		\Big)^{\frac{1}{2}}\left(\int_{\Omega}	\langle \nabla{f}, T(\nabla{f}) \rangle^{2} u_j^{2}\ \mathrm{dm} \right)^{\frac{1}{2}}\!+\! \int_{\Omega}\! \Big(2\langle T(\nabla{u_j}), \nabla{f} \rangle + u_j\mathcal{L}{f}\Big)^{2}\! \mathrm{dm} \\
		&+ \left(\int_{\Omega}	\langle \nabla{f}, T(\nabla{f}) \rangle^{2} u_j^{2}\ \mathrm{dm} \right)^{\frac{1}{2}}\Big((\lambda_{k+2} - \lambda_j)(\lambda_{k+1} - \lambda_j)\Big)^\frac{1}{2} \\[10pt]
		=& 2\sqrt{\!(\lambda_{k+2} \!-\! \lambda_j)(\lambda_{k+1} \!-\! \lambda_j)\int_{\Omega}	\langle \nabla{f}, T(\nabla{f}) \rangle^{2} u_j^{2}\ \mathrm{dm}} 
		\!+\! \int_{\Omega}\! \Big(\! 2\langle T(\nabla{u_j}),\! \nabla{f} \rangle \!+\! u_j\mathcal{L}{f}\Big)^{2} \mathrm{dm},
    \end{split}
	\end{equation}}
    this concludes the proof of Corollary~\ref{C3.1}.
\end{proof}

From Corollary~\ref{C3.1} we obtain the following result that will be useful to prove our main theorems.
	
\begin{cor}\label{C3.2}
    Under the same setup as in Lemma~\ref{L3.2}, for any real-valued function $f\in C^3(\Omega)\cap C^{2}(\overline{\Omega})$ with $|\nabla{f}|^{2}=1$, we have 
	{\small
	\begin{equation}\label{3.14}
		(\lambda_{k+2} \!-\! \lambda_{k+1})^{2} \!\leq\! 
		\frac{16}{\sigma}\! \left(\!\int_{\Omega}\!\langle T(\nabla{u_{j}}),\! \nabla{f} \rangle^{2} \mathrm{dm} \!-\! \frac{1}{4}\!\int_{\Omega}\!(\mathcal{L}{f})^{2}u_{j}^{2} \mathrm{dm}
		\!-\! \frac{1}{2}\!\int_{\Omega}\!
		\langle\nabla(\mathcal{L}{f}),\!  T(\nabla{f})\rangle u_{j}^{2} \mathrm{dm}\!\right)\!\lambda_{k+2}.
	\end{equation}}
	Moreover,
	{\small
	\begin{equation}\label{3.15}
		\lambda_{k+2} - \lambda_{k+1} \leq \frac{4}{\sqrt{\sigma}}\left(\delta\lambda_{j} - \frac{1}{4}\int_{\Omega}(\mathcal{L}{f})^{2}u_{j}^{2} \mathrm{dm} - \frac{1}{2}\int_{\Omega}\langle\nabla(\mathcal{L}{f}),\!  T(\nabla{f})\rangle u_{j}^{2} \mathrm{dm}\right)^\frac{1}{2}\sqrt{\lambda_{k+2}},
	\end{equation}}
	where $\sigma =2\delta -\varepsilon$, with $\varepsilon$ and $\delta$ given by \eqref{2.1}.
\end{cor}

\begin{proof}
    Since $|\nabla{f}|^{2} = 1$, from Corollary~\ref{C3.1} we have
	\begin{equation}
    \nonumber
    \begin{split}
		\varepsilon\Big(\!(\!\lambda_{k+2} \!-\! \lambda_{j}\!)\! +\! (\!\lambda_{k+1} \!-\! \lambda_{j}\!)\!\Big)\int_{\Omega} u_j^{2}\ \mathrm{dm} &\leq   2\delta\sqrt{\!\big(\!\lambda_{k+2} \!-\! \lambda_{j}\!\big)\!\big(\!\lambda_{k+1} \!-\! \lambda_{j}\!\big)\int_{\Omega}u_j^{2}\ \mathrm{dm}} \\
		&\quad + \int_{\Omega}\Big(2\langle T(\nabla{u_{j}}), \nabla{f} \rangle + u_{j}\mathcal{L}{f}\Big)^{2} \mathrm{dm},
    \end{split}
	\end{equation}
	that is
	{\small\begin{equation}
        \nonumber
		\varepsilon\Big(\!(\lambda_{k+2} \!-\! \lambda_{j}) \!+\! (\lambda_{k+1} \!-\! \lambda_{j})\!\Big) - 2\delta\sqrt{\big(\lambda_{k+2} \!-\! \lambda_{j}\big)\big(\lambda_{k+1} \!-\! \lambda_{j}\big)} \!\leq\! \int_{\Omega}\!\Big(\!2\langle T(\nabla{u_{j}}),\! \nabla{f} \rangle \!+\! u_{j}\mathcal{L}{f}\Big)^{2} \mathrm{dm}.
	\end{equation}}
    Taking $\sigma = 2\delta -\varepsilon$ we get
	\begin{equation}
    \nonumber
		\frac{\sqrt{(\lambda_{k+2} - \lambda_{j})(\lambda_{k+1} - \lambda_{j})}}{(\lambda_{k+2} - \lambda_{j}) + (\lambda_{k+1} - \lambda_{j})} \leq \frac{\varepsilon - \sigma}{2(\delta - \sigma)} ,
	\end{equation}
    so we obtain
	{\small
	\begin{equation}
    \nonumber
    \begin{split}
		\sigma \Big(\!\sqrt{\lambda_{k+2} \!-\! \lambda_{j}} &\!-\! \sqrt{\lambda_{k+1} \!-\! \lambda_{j}}\!\Big)^{2}\! \leq\! \varepsilon\Big(\!(\lambda_{k+2} \!-\! \lambda_{j}) \!+\! (\lambda_{k+1} \!-\! \lambda_{j})\!\Big) \!-\! 2\delta \sqrt{\big(\!\lambda_{k+2} \!-\! \lambda_{j}\!\big)\big(\!\lambda_{k+1} \!-\! \lambda_{j}\!\big)} \\[12pt]
		&\leq\! \int_{\Omega}\!\Big(2\langle T(\nabla{u_{j}}), \nabla{f} \rangle + u_{j}\mathcal{L}{f}\Big)^{2} \mathrm{dm} \\[12pt]
		&=\! \int_{\Omega}\!\Big(\!4\langle T(\nabla{u_{j}}), \nabla{f} \rangle^{2} \!+\! (\mathcal{L}{f})^{2}u_{j}^{2} \!+\! 4\langle T(\nabla{u_{j}}),\! \nabla{f} \rangle u_{j}\mathcal{L}{f}\!\Big) \mathrm{dm} \\[12pt]
		&=\! 4\!\int_{\Omega}\!\langle T(\nabla{u_{j}}),\! \nabla{f} \rangle^{2} \mathrm{dm} \!+\! \int_{\Omega}\!(\mathcal{L}{f})^{2}u_{j}^{2} \mathrm{dm} \!+\! 4\!\int_{\Omega}\!\langle T(\nabla{u_{j}}),\! \nabla{f} \rangle u_{j}\mathcal{L}{f} \mathrm{dm}.
    \end{split}
	\end{equation}}
	    Applying the Green's formula and since that $f$ is a real-valued function, we have
	\begin{equation}
    \nonumber
    \begin{split}
		4\int_{\Omega}\langle T(\nabla{u_{j}}), \nabla{f} \rangle u_{j}\mathcal{L}{f} \mathrm{dm}
		&= 2\int_{\Omega}\langle T(\nabla{u_{j}^{2}}), (\mathcal{L}{f})\nabla{f} \rangle \mathrm{dm} =  2\int_{\Omega} \langle \nabla{u_{j}^{2}},T((\mathcal{L}{f})\nabla{f})\rangle \mathrm{dm} \\
		&= - 2\int_{\Omega}u_{j}^{2}(\mathrm{div}_{\eta}[T((\mathcal{L}{f}) \nabla{f})]) \mathrm{dm} \\ 
		& = - 2\int_{\Omega}u_{j}^{2} \Big(\mathcal{L}{f} (\mathrm{div}_{\eta}[T(\nabla{f})]) + \langle T(\nabla{f}), \nabla(\mathcal{L}{f}) \rangle\Big) \mathrm{dm} \\
		& = - 2\int_{\Omega}u_{j}^{2}(\mathcal{L}{f})^{2} \mathrm{dm} - 2\int_{\Omega}u_{j}^{2}\langle\nabla(\mathcal{L}{f}),  T(\nabla{f})\rangle \mathrm{dm}.
    \end{split}
	\end{equation}	
    Thus
		\begin{equation}
	\nonumber
    \begin{split}
		\sigma \Big(\!\sqrt{\!\lambda_{k+2} \!-\! \lambda_{j}}& \!-\! \sqrt{\!\lambda_{k+1} \!-\! \lambda_{j}}\!\Big)^{2} \\
        \leq & 4\!\int_{\Omega}\!\langle T(\nabla{u_{j}}),\! \nabla{f} \rangle^{2} \mathrm{dm} \!+\! \int_{\Omega}\!(\mathcal{L}{f})^{2}u_{j}^{2} \mathrm{dm} \!+\! 4\!\int_{\Omega}\!\langle T(\nabla{u_{j}}),\! \nabla{f} \rangle u_{j}\mathcal{L}{f} \mathrm{dm} \\
		=& 4\int_{\Omega}\langle T(\nabla{u_{j}}), \nabla{f} \rangle^{2} \mathrm{dm} \!+\! \int_{\Omega}(\mathcal{L}{f})^{2}u_{j}^{2} \mathrm{dm} - 2\int_{\Omega}
		(\mathcal{L}{f})^{2}u_{j}^{2} \mathrm{dm} \\
  &- 2\int_{\Omega}\langle\nabla(\mathcal{L}{f}),  T(\nabla{f})\rangle u_{j}^{2} \mathrm{dm} \\
		=& 4\!\int_{\Omega}\!\langle T(\nabla{u_{j}}),\! \nabla{f} \rangle^{2} \mathrm{dm} \!-\! \int_{\Omega}\!
		(\!\mathcal{L}{f}\!)^{2}u_{j}^{2}  \mathrm{dm} \!-\! 2\!\int_{\Omega}\!
		\langle\nabla(\mathcal{L}{f}),\!  T(\nabla{f})\rangle u_{j}^{2} \mathrm{dm}.
    \end{split}
	\end{equation}
    Multiplying the previous inequality by $\dfrac{1}{\sigma}\left(\sqrt{\lambda_{k+2} - \lambda_{j}} + \sqrt{\lambda_{k+1} - \lambda_{j}}\right)^{2}$ we find
	{\small
    \begin{equation}
    \nonumber
    \begin{split}
		\Big(\!\lambda_{k+2} \!-\! \lambda_{k+1}\!\Big)^{2}
		\!&\leq\! \frac{4}{\sigma}\! \left(\int_{\Omega}\langle T(\nabla{u_{j}}), \nabla{f} \rangle^{2} \mathrm{dm} - \frac{1}{4}\int_{\Omega}(\mathcal{L}{f})^{2}u_{j}^{2} \mathrm{dm} \right. \\
		&\quad \left. - \frac{1}{2}\int_{\Omega}
		\langle\nabla(\mathcal{L}{f}),  T(\nabla{f})\rangle u_{j}^{2}\mathrm{dm}\right)\!\bigg(\sqrt{\lambda_{k+2} - \lambda_{j}} + \sqrt{\lambda_{k+1} - \lambda_{j}}\bigg)^{2} \\[12pt]
		&\leq\! \frac{16}{\sigma}\! \left(\!\int_{\Omega}\!\langle T(\nabla{u_{j}}),\! \nabla{f} \rangle^{2} \mathrm{dm} \!-\! \frac{1}{4}\!\int_{\Omega}\! (\mathcal{L}{f})^{2}u_{j}^{2} \mathrm{dm}
		\!-\! \frac{1}{2}\!\int_{\Omega}\!
		\langle\nabla(\mathcal{L}{f}),\!  T(\nabla{f})\rangle u_{j}^{2} \mathrm{dm}\!\right)\!\lambda_{k+2},
    \end{split}
	\end{equation}}
    which completes the proof of inequality~\eqref{3.14}, once that $\lambda_{k+1} \leq \lambda_{k+2}$ implies 
	\begin{equation}
    \nonumber
    \begin{split}
		\left(\sqrt{\lambda_{k+2} - \lambda_{j}} + \sqrt{\lambda_{k+1} - \lambda_{j}}\right)^{2}\! &\leq\! \left(\sqrt{\lambda_{k+2} - \lambda_{j}} + \sqrt{\lambda_{k+2} - \lambda_{j}}\right)^{2} \!=\! \left(2\sqrt{\lambda_{k+2} - \lambda_{j}} \right)^{2} \\
		&= 4(\lambda_{k+2} - \lambda_{j}) \leq 4\lambda_{k+2}.
    \end{split}
	\end{equation}
	
    Now, to prove \eqref{3.15} observe that from \eqref{T-property} and $|\nabla{f}|^{2} = 1$ we have
    \begin{equation}\label{3.16}
		\langle T(\nabla{u_{j}}),\nabla{f}\rangle^{2} \leq  |T(\nabla{u_{j}})|^{2}|\nabla{f}|^{2}
		=  |T(\nabla{u_{j}})|^{2} \leq \delta\langle\nabla{u_{j}}, T(\nabla{u_{j}})\rangle.
	\end{equation}
    Finally, applying \eqref{3.16} into \eqref{3.14} and remembering that $\displaystyle\lambda_{j}=\int_{\Omega}\langle\nabla{u_{j}}, T(\nabla{u_{j}})\rangle \mathrm{dm}$  we get
	{\small
    \begin{equation}
    \nonumber
    \begin{split}
		\Big(\!\lambda_{k+2} \!-\! \lambda_{k+1}\!\Big)^{2} &\leq\! \frac{16}{\sigma}\! \left(\!\int_{\Omega}\!\langle\nabla{f}, T(\nabla{u_{j}})\rangle \mathrm{dm} \!-\! \frac{1}{4}\!\int_{\Omega}\! (\mathcal{L}{f})^{2}u_{j}^{2} \mathrm{dm}
		\!-\! \frac{1}{2}\!\int_{\Omega}\!
		\langle\nabla(\mathcal{L}{f}),\!  T(\nabla{f})\rangle u_{j}^{2} \mathrm{dm}\!\right)\!\lambda_{k+2} \\[8pt]
		&\leq\! \frac{16}{\sigma}\! \left(\!\delta\int_{\Omega}\!\langle\nabla{u_{j}}, T(\nabla{u_{j}})\rangle \mathrm{dm} \!-\! \frac{1}{4}\!\int_{\Omega}\! (\mathcal{L}{f})^{2}u_{j}^{2} \mathrm{dm}
		\!-\! \frac{1}{2}\!\int_{\Omega}\!
		\langle\nabla(\mathcal{L}{f}),\!  T(\nabla{f})\rangle u_{j}^{2} \mathrm{dm}\!\right)\!\lambda_{k+2} \\[8pt]
		&= \frac{16}{\sigma}\!\left(\delta\lambda_{j} - \frac{1}{4}\int_{\Omega}(\mathcal{L}{f})^{2}u_{j}^{2} \mathrm{dm} - \frac{1}{2}\int_{\Omega}\langle\nabla(\mathcal{L}{f}),\!  T(\nabla{f})\rangle u_{j}^{2} \mathrm{dm}\right)\lambda_{k+2},
    \end{split}
	\end{equation}}
	this is sufficient to prove expression~\eqref{3.15}.
\end{proof}

\section{Proof of Theorems~\ref{T1.1}, \ref{T1.2} and \ref{T1.3}}

In addition to the auxiliary results, to proof the main results, we will make use of the following fact
\begin{lem}[\cite{CJM}]\label{Yang-Ineq1}
    Let $\Omega$ be a bounded domain in an $n$-dimensional complete Riemannian manifold $M^n$ isometrically immersed in $\mathbb{R}^m$, and $\lambda_{i}$ be the $i$-th eigenvalue of Problem~\eqref{problem1}. Then, we have
    \begin{equation}
        \nonumber
        \upsilon_{k+1} \leq \left(1+\frac{4\delta}{n\varepsilon}\right)k^{\frac{2\delta}{n\varepsilon}}\upsilon_1,
    \end{equation}
    where
    \begin{equation}
        \nonumber
        \upsilon_{k+1}=\lambda_{k+1}+\frac{n^2H_0^2+4C_0+T_0^2}{4\delta},
    \end{equation}
    with  $H_0$, $C_{0}$ and $T_{0}$ are given as in Theorem~\ref{T1.2}.
%    \begin{equation}
%    C_0=\sup_\Omega \left\{\frac{1}{2}\dv \Big( T\big(T(\nabla \eta)-\tr(\nabla T)\big)\Big) - \frac{1}{4}|T(\nabla \eta)|^2\right\},
%    \end{equation}
    %$T_0=\sup_{\Omega}|\tr(\nabla T)|$, $\eta_0=\sup_{\Omega}|\nabla \eta|$ and $H_0=\sup_\Omega |{\bf H}_T|$. 
\end{lem} 
    
The first result to be presented states that Conjecture~\ref{Cj2} is true in the Euclidean case.  We emphasize that the main tool in the proof is Corollary~\ref{C3.2}.
 
\subsection{Proof of Theorem~\ref{T1.1}}
\begin{proof}
	Without loss of generality, we assume that $\lambda_{k+1}<\lambda_{k+2}$. Let $\{x_1,\cdots,x_n\}$ be the standard coordinate functions in $\mathbb{R}^{n}$. Since $\nabla{x_l} = e_l$ for all $l=1,\cdots, n$, we have $|\nabla{x_l}| = 1$, taking $j=1$ and $f=x_l$ in \eqref{3.14}. Thus
	\begin{equation}
    \nonumber
    \begin{split}
		(\lambda_{k+2} - \lambda_{k+1})^{2}\! &= \frac{16}{\sigma} \left(\int_{\Omega} \!\langle{e_l}, T(\nabla{u_1})\rangle^{2} \mathrm{dm} \!-\! \dfrac{1}{4} \int_{\Omega}\! \big(\mathrm{div}(T(e_l)) - \langle \nabla{\eta}, T(e_l)\rangle\big)^{2}u_1^{2} \mathrm{dm} \right. \\ 
		& \left.\quad \ - \dfrac{1}{2} \int_{\Omega} \!\langle\nabla\big(\mathrm{div}(T(e_l)) - \langle \nabla{\eta}, T(e_l)\rangle\big), T({e_l})\rangle u_1^{2} \mathrm{dm}\right)\!\lambda_{k+2},
    \end{split}
	\end{equation}
    and summing over l from 1 to n, we have 
	{\small
	\begin{equation}\label{4.1}
	\begin{split}
		n(\lambda_{k+2} - \lambda_{k+1})^{2}\! &\leq \frac{16}{\sigma} \left(\int_{\Omega} \! \sum_{l=1}^{n}\langle{e_l}, T(\nabla{u_1})\rangle^{2} \mathrm{dm} \!-\! \dfrac{1}{4} \int_{\Omega}\! \sum_{l=1}^{n}\big(\mathrm{div}(T(e_l) - \langle T(\nabla{\eta}), e_l\rangle\big)^{2}u_1^{2} \mathrm{dm} \right. \\ 
		& \left.\quad \ - \dfrac{1}{2} \int_{\Omega} \!\sum_{l=1}^{n}\langle\nabla\big(\mathrm{div}(T(e_l) - \langle T(\nabla{\eta}), e_l\rangle\big), T({e_l})\rangle u_1^{2} \mathrm{dm}\right)\!\lambda_{k+2}.
	\end{split}
	\end{equation}}
    Now, we will use \eqref{T-property} to obtain
	\begin{equation}\label{4.2}
		\int_{\Omega} \! \sum_{l=1}^{n}\langle{e_l}, T(\nabla{u_1})\rangle^{2} \mathrm{dm} = \int_{\Omega} \! |T(\nabla{u_1})|^{2} \mathrm{dm} \leq  \delta \int_{\Omega} \! \langle T(\nabla{u_1}), \nabla{u_1} \rangle \mathrm{dm} = \delta\lambda_{1}.
	\end{equation}
	Note that
    {\small
	\begin{equation}
        \nonumber
		\mathrm{div}(T(e_l)) = \sum_{j=1}^{n}\left\langle\nabla_{e_j}T(e_l), e_j\right\rangle = \sum_{j=1}^{n}\left\langle(\nabla_{e_j}T)(e_l), e_j\right\rangle = \left\langle e_l, \sum_{j=1}^{n}(\nabla_{e_j}T)e_j\right\rangle = \left\langle e_l, \mathrm{tr}(\nabla T) \right\rangle,
	\end{equation}}
	then
	\begin{equation}\label{4.3}
	\begin{split}
		&- \dfrac{1}{4} \int_{\Omega}\! \sum_{l=1}^{n}\big(\mathrm{div}(T(e_l)) - \langle T(\nabla{\eta}), e_l\rangle\big)^{2}u_1^{2} \mathrm{dm} \\
		&= - \dfrac{1}{4} \int_{\Omega}\! \left(\sum_{l=1}^{n} \left\langle e_l, \mathrm{tr}(\nabla T) \right\rangle^{2} - 2\mathrm{div}\left(T\left(\sum_{l=1}^{n}\left\langle T(\nabla{\eta}), e_l\right\rangle e_l\right)\right) \right.\\
		&\quad \left. + 2\sum_{l=1}^{n}\left\langle \nabla\big(\left\langle T(\nabla{\eta}), e_l\right\rangle\big), T(e_l)\right\rangle +  |T(\nabla{\eta})|^{2}\right)u_1^{2} \mathrm{dm} \\
		&= - \dfrac{1}{4} \int_{\Omega}\! \left(|\mathrm{tr}(\nabla T)|^{2} - 2\mathrm{div}\left(T^{2}(\nabla{\eta})\right)  \right.\\
		&\quad \left. + 2\sum_{j=1}^{n}\left\langle \nabla_{e_{j}} T(\nabla{\eta}), \sum_{l=1}^{n} \left\langle T({e_j}), e_{l}\right\rangle e_l\right\rangle +  |T(\nabla{\eta})|^{2}\right)u_1^{2} \mathrm{dm} \\
		&= - \frac{1}{4} \int_{\Omega}\! \left(|\mathrm{tr}(\nabla T)|^{2}\! -\!2\mathrm{div}\left(T^{2}(\nabla{\eta})\right) \!+\! 2\sum_{j=1}^{n}\!\left\langle \nabla_{e_{j}} T(\nabla{\eta}), T({e_j})\right\rangle \!+\!  |T(\nabla{\eta})|^{2}\!\right)\!u_1^{2} \mathrm{dm}.
	\end{split}
	\end{equation}
	Furthermore, we have 
	{\small
	\begin{equation}\label{4.4}
	\begin{split}
		&- \dfrac{1}{2} \int_{\Omega} \!\sum_{l=1}^{n}\langle\nabla\big(\mathrm{div}(T(e_l) - \langle T(\nabla{\eta}), e_l\rangle\big), T({e_l})\rangle u_1^{2} \mathrm{dm} \\
		=&- \dfrac{1}{2} \int_{\Omega} \!\sum_{l=1}^{n}\left( \left\langle \sum_{j=1}^{n}\left\langle e_l, \nabla_{e_{j}} \mathrm{tr}(\nabla T) \right\rangle e_{j}, T({e_l})\right\rangle - \left\langle\sum_{j=1}^{n} \left\langle \nabla_{e_{j}}T(\nabla{\eta}), e_l\right\rangle e_{j}, T({e_l})\right\rangle\right) u_1^{2} \mathrm{dm} \\
		=&- \dfrac{1}{2} \int_{\Omega} \!\left( \sum_{j=1}^{n}\left\langle T({e_j}), \nabla_{e_{j}} \mathrm{tr}(\nabla T) \right\rangle - \sum_{j=1}^{n}\left\langle \nabla_{e_{j}} T(\nabla{\eta}), T({e_j})\right\rangle \right) u_1^{2} \mathrm{dm}.
	\end{split}
	\end{equation}}
    Since
    {\small\begin{equation}
    \nonumber
    \begin{split}
		&\mathrm{div}\Big(T\big(\mathrm{tr}(\nabla T)\big)\Big) = \sum_{j=1}^{n}\left\langle \nabla_{e_{j}} T(\mathrm{tr}(\nabla T)), {e_j}\right\rangle = \sum_{j=1}^{n}\left\langle (\nabla_{e_{j}} T)(\mathrm{tr}(\nabla T)) + T(\nabla_{e_{j}}\mathrm{tr}(\nabla T)), {e_j}\right\rangle \\
		&= \left\langle \mathrm{tr}(\nabla T), \sum_{j=1}^{n} (\nabla_{e_{j}} T){e_j}\right\rangle + \sum_{j=1}^{n}\left\langle \nabla_{e_{j}}\mathrm{tr}(\nabla T), T({e_j})\right\rangle = |\mathrm{tr}(\nabla T)|^{2} + \sum_{j=1}^{n}\left\langle \nabla_{e_{j}}\mathrm{tr}(\nabla T), T({e_j})\right\rangle,
    \end{split}
	\end{equation}}
	and from the expressions \eqref{4.1}, \eqref{4.2}, \eqref{4.3} and \eqref{4.4}, it follows that
	\begin{equation}
    \nonumber
    \begin{split}
		(\lambda_{k+2} - \lambda_{k+1})^{2}\! &\leq \frac{16}{\sigma n}\left(\! \delta\lambda_{1} + \int_{\Omega}\!u_1^{2} \bigg[- \dfrac{1}{4}|\mathrm{tr}(\nabla T)|^{2} + \dfrac{1}{2}\mathrm{div}\left(T^{2}(\nabla{\eta})\right)  \right. \\
		&\quad \ - \dfrac{1}{2}\sum_{j=1}^{n}\left\langle \nabla_{e_{j}} T(\nabla{\eta}), T({e_j})\right\rangle - \dfrac{1}{4}|T(\nabla{\eta})|^{2} \\
		& \left.\quad \ - \dfrac{1}{2}\sum_{j=1}^{n}\left\langle T({e_j}), \nabla_{e_{j}} \mathrm{tr}(\nabla T) \right\rangle + \dfrac{1}{2}\sum_{j=1}^{n}\left\langle \nabla_{e_{j}} T(\nabla{\eta}), T({e_j})\right\rangle \bigg] \mathrm{dm}\right)\!\lambda_{k+2} \\
		&= \frac{16}{\sigma n}\left(\! \delta\lambda_{1} + \int_{\Omega}\!u_1^{2} \bigg[ \frac{1}{2}\mathrm{div}
		\bigg(T\Big(T(\nabla\eta) - \mathrm{tr}(\nabla T)\Big)\bigg) - \frac{1}{4}|T(\nabla\eta)|^{2} \right.  \\
		& \left.\quad \ + \dfrac{1}{4}|\mathrm{tr}(\nabla T)|^{2} \bigg] \mathrm{dm}\right)\!\lambda_{k+2} \\
		&\leq \frac{16}{\sigma n}\left(\! \delta\lambda_{1} + \int_{\Omega}\!u_1^{2} \bigg[ C_0 + \dfrac{T_{0}^{2}}{4} \bigg] \mathrm{dm}\right)\!\lambda_{k+2} = \frac{16}{\sigma n}\left(\! \delta\lambda_{1} + C_0 + \dfrac{T_{0}^{2}}{4}\right)\!\lambda_{k+2}.
    \end{split}
	\end{equation}
	Consequently, from Lemma~\ref{Yang-Ineq1} and the fact that $H_{0}^{2} = 0$ in $\mathbb{R}^n$, we conclude that
	\begin{equation}\label{4.5}
    \begin{split}
        \lambda_{k+2} - \lambda_{k+1} &\leq 4\sqrt{\frac{1}{\sigma n}\left(\! \delta\lambda_{1} + C_0 + \dfrac{T_{0}^{2}}{4}\right)\!\lambda_{k+2}} \\
		&\leq 4\sqrt{\frac{1}{\sigma n}\left(\! \delta\lambda_{1} + C_0 + \dfrac{T_{0}^{2}}{4}\right)\!\left(1 + \frac{4\delta}{n\varepsilon}\right)\!\left(\!\lambda_{1} + \dfrac{4C_0 + T_{0}^{2}}{4\delta}\right)\big(k+1\big)^{\frac{2\delta}{n\varepsilon}}} \\
		&= 4\sqrt{\frac{\delta}{\sigma n}\!\left(1 + \frac{4\delta}{n\varepsilon}\right)}
		\sqrt{\left(\!\lambda_{1} + \dfrac{4C_0 + T_{0}^{2}}{4\delta}\right)^{2}\big(k+1\big)^{\frac{2\delta}{n\varepsilon}}} \\
		&= 4\sqrt{\frac{\delta}{\sigma n}\!\left(1 + \frac{4\delta}{n\varepsilon}\right)}
		\left(\!\lambda_{1} + \dfrac{4C_0 + T_{0}^{2}}{4\delta}\right)\big(k+1\big)^{\frac{\delta}{n\varepsilon}} \\
		&= C_{n,\Omega}\big(k+1\big)^{\frac{\delta}{n\varepsilon}},
    \end{split}
	\end{equation} 
	where $\displaystyle C_{n,\Omega} = 4\left(\!\lambda_{1} + \dfrac{4C_0 + T_{0}^{2}}{4\delta}\right)\sqrt{\frac{\delta}{\sigma n}\!\left(1 + \frac{4\delta}{n\varepsilon}\right)}$. Therefore, \eqref{4.5} is true for arbitrary $k>1$.
\end{proof}

\subsection{Proof of Theorem~\ref{T1.2}}
\begin{proof}
We will use the upper-half-plane model of Hyperbolic space, that is,
	\begin{equation}
        \nonumber
		\mathbb{H}^n(-1) = \{(x_1,\cdots, x_n) \in \mathbb{R}^{n}; x_n>0\},
	\end{equation}
	equipped with the metric
	\begin{equation}
        \nonumber
		g_{js}(x_1, \cdots, x_n) = \left\langle \partial_{j}, \partial_{s} \right\rangle = \frac{1}{(x_n)^{2}} \delta_{js},
	\end{equation}
	where $\{\partial_{j}\}_{j=1}^n$ is the basis of the coordinate vector fields of $\mathbb{H}^n$.
		
	Taking $f=\ln{x_n}$ we will have
		%\allowdisplaybreaks
	\begin{equation}\label{4.6}
		\nabla(\ln x_n) = x_n\partial_n, \quad |\nabla{(\ln x_n)}| = 1\quad \mbox{ and } \quad \Delta(\ln x_n) = -(n-1),
	\end{equation}
    so, make use of the hypothesis that $T(\partial_{n})=\psi\partial_{n}$ and the fact that $\eta$ and $\psi$ are radially constant, we compute
	\begin{equation}\label{4.7}
	\begin{split}
	    \mathcal{L}{(\ln{x_n})} &= \mathrm{div}(T(\nabla(\ln x_n))) -  \left\langle \nabla\eta, T(\nabla{(\ln{x_n})})\right\rangle \\
		&= \mathrm{div}(T(x_{n}\partial_{n})) -  \left\langle \nabla\eta, T(x_{n}\partial_{n})\right\rangle \\
		&= \mathrm{div}(\psi(x_{n}\partial_{n})) -  \left\langle \nabla\eta, \psi(x_{n}\partial_{n})\right\rangle \\
		&= \psi \Delta(\ln x_n) + \left\langle \nabla\psi, x_{n}\partial_{n}\right\rangle - \psi\left\langle \nabla\eta, x_{n}\partial_{n}\right\rangle \\
		&= -(n-1)\psi.
	\end{split}
	\end{equation}
		
	We may assume without loss of generality that $\lambda_{k+1}<\lambda_{k+2}$. From \eqref{4.7}, \eqref{4.6} and taking $j=1$ in \eqref{3.15}, we obtain 
	{\small
	\begin{equation}
    \nonumber
    \begin{split}
        \lambda_{k+2} \!-\! \lambda_{k+1}\! &\leq\! \frac{4}{\sqrt{\sigma}}\!\left(\!\delta\lambda_{1} \!-\! \dfrac{1}{4}\!\int_{\Omega}\!
		\big(\!\mathcal{L}\!{(\ln{x_n})}\!\big)^{2}\!u_{1}^{2} \mathrm{dm} \!-\! \dfrac{1}{2}\!\int_{\Omega}\!
		\Big\langle\!\nabla(\!\mathcal{L}\!{(\ln{x_n})}\!),\! T(\!\nabla{\!(\ln{x_n}\!)}\!)\!\Big\rangle\! u_{1}^{2} \mathrm{dm}\!\right)^\frac{1}{2}\!\sqrt{\!\lambda_{k+2}} \\
		&= \frac{4}{\sqrt{\sigma}}\left(\delta\lambda_{1} - \dfrac{1}{4}	\big(n-1\big)^{2}\int_{\Omega}\psi^{2} u_{1}^{2} \mathrm{dm} \right)^\frac{1}{2}\sqrt{\lambda_{k+2}},
    \end{split}
	\end{equation}}
	and since $\varepsilon\leq\left\langle T(\nabla\ln x_n),\nabla\ln x_n\right\rangle=\psi$ we have
	\begin{equation}\label{4.8}
		\lambda_{k+2}\!-\!\lambda_{k+1}\!\leq \frac{4}{\sqrt{\sigma}}\left(\delta\lambda_{1} - \dfrac{\varepsilon^{2}}{4}
		\big(n-1\big)^{2}\right)^\frac{1}{2}\sqrt{\lambda_{k+2}}.
	\end{equation}
	According to Lemma~\ref{Yang-Ineq1} we have
	\begin{equation}\label{4.9}
		\lambda_{k+1} \leq \left(1 + \frac{4\delta}{n\varepsilon}
		\right) \left(\lambda_{1} + \frac{n^{2}H_{0}^{2} + 4C_{0} + T_{0}^{2}}{4\delta}\right)k^{\frac{2\delta}{n\varepsilon}},
	\end{equation}
	then we can apply \eqref{4.9} to \eqref{4.8}, and conclude 
    {\small
	\begin{equation}
	\nonumber
    \begin{split}
		\lambda_{k+2}\!-\!\lambda_{k+1} &\leq 
		\frac{4}{\sqrt{\sigma}}\! \left(\!\delta\lambda_{1} \!-\! \dfrac{\varepsilon^{2}}{4}\!\big(n\!-\!1\big)^{2}\!\right)^\frac{1}{2} \!\sqrt{\left(1 + \frac{4\delta}{n\varepsilon} \right) \left(\lambda_{1} + \frac{n^{2}H_{0}^{2} + 4C_{0} + T_{0}^{2}}{4\delta}\right) \Big(k+1\Big)^{\frac{2\delta}{n\varepsilon}}} \\
		&= C_{n,\Omega} \Big(k+1\Big)^{\frac{\delta}{n\varepsilon}},\quad \mbox{ for any } k>1,
    \end{split}
	\end{equation}}
	where $\displaystyle C_{n,\Omega} = \frac{4}{\sqrt{\sigma}}\sqrt{\!\left(1 + \frac{4\delta}{n\varepsilon}\right)\!\left(\!\delta\lambda_{1} \!-\! \dfrac{\varepsilon^{2}}{4}\!\big(n\!-\!1\big)^{2}\!\right)\!\left(\lambda_{1} + \frac{n^{2}H_{0}^{2}\!+\!4C_{0}\!+\!T_{0}^{2}}{4\delta}\right)}$. 
\end{proof}

\subsection{Proof of Theorem~\ref{T1.3}}
\begin{proof}
    From inequality~\eqref{3.15} of the Corollary~\ref{C3.2} for $j=1$ and the distance function $f=r$, since $|\nabla{r}|^{2}=1$ we obtain 
	\begin{equation}\label{4.10}
        \lambda_{k+2} - \lambda_{k+1} \leq \frac{4}{\sqrt{\sigma}}\left(\delta\lambda_{1} + \frac{1}{4}\int_{\Omega}\left(-(\square_{\eta}{r})^{2} - 2\langle\nabla(\square_{\eta}{r}),\!T(\partial{r})\rangle \right)u_{1}^{2} \mathrm{dm}\right)^\frac{1}{2}\sqrt{\lambda_{k+2}}.
	\end{equation}
    We need to estimate the expression 
    \begin{equation}
        \nonumber
        -(\square_{\eta}r)^2-2\langle\nabla(\square_{\eta}r),T(\partial_{r})\rangle.
    \end{equation}
    For this, we use that the tensor $T$ is radially parallel, that is, $\nabla_{\partial_{r}}T$ is null, such that, from Bochner formula~\eqref{Bo} for drifted Cheng-Yau operator, we get
	\begin{equation}
    \nonumber
        -\langle\nabla(\square_{\eta}r), T(\partial_{r})\rangle=-\psi_{n}\langle\nabla(\square_{\eta}r), \partial_{r}\rangle=\psi_{n}R_{\eta,T}(\partial_{r},\partial_{r})+\psi_{n}\langle\nabla^{2}r,\nabla^{2}r\circ{T}\rangle.
    \end{equation}
    Furthermore, we have
    \begin{equation}\label{4.11}
    \begin{split}
        -(\square_{\eta}{r})^{2} - 2\langle\nabla(\square_{\eta}{r}),\!T(\partial_{r})\rangle = &-\left(\square{r}\right)^{2}+2\langle\nabla\eta, T(\partial_{r})\rangle\square{r}-\langle\nabla\eta, T(\partial_{r})\rangle^{2}\\ 
        &+ 2\psi_{n}R_{\eta,T}(\partial_{r},\partial_{r})\quad + 2\psi_{n}\langle\nabla^{2}{r}, \nabla^{2}{r}\circ T\rangle.
    \end{split}
    \end{equation}
    Let us consider an orthonormal basics $\{\partial_{r}, e_{1}, \cdots, e_{n-1}\}$ of the eigenvectors of the operator $\nabla^{2}{r}$ with $0\leq h_{1} \leq \cdots \leq h_{n-1}$ the associated eigenvalues, respectively.
    Using the diagonalized form of the matrix associated with the operator $\nabla^{2}{r}$ we can calculate
    \begin{equation}
    \nonumber
    \begin{split}
        \langle\nabla^{2}{r}, \nabla^{2}{r} \circ T\rangle & =\sum_{j=1}^{n-1} \langle\nabla^{2}{r}(e_{j}), \nabla^{2}{r} \circ T(e_{j})\rangle=\sum_{j=1}^{n-1}  h_{j}\langle e_{j}, \nabla^{2}{r} \circ T(e_{j})\rangle\\
        &=\sum_{j=1}^{n-1} h_{j}\langle \nabla^{2}{r}(e_{j}), T(e_{j})\rangle =\sum_{j=1}^{n-1} h_{j}^{2}\langle e_{j},T(e_{j})\rangle=\sum_{j=1}^{n-1} h_{j}^{2} \theta_{j},
    \end{split}
    \end{equation}
    where $\theta_{j}=\langle e_{j},T(e_{j})\rangle$, and also
    \begin{equation}
        \nonumber
        \square{r}=\operatorname{tr}\!\left(\nabla^{2}{r},T\right)=\sum_{j=1}^{n-1} \langle\nabla^{2}{r}(e_{j}),T(e_{j})\rangle =\sum_{j=1}^{n-1} h_{j} \theta_{j}.
    \end{equation}
    Since $\varepsilon\left|X\right|^{2} \leq\langle T(X),X\rangle\leq\delta\left|X\right|^{2}$, we have
    \begin{equation}
    \nonumber
    \begin{split}
        &2\psi_{n}\langle\nabla^{2}{r}, \nabla^{2}{r}\circ T\rangle-(\square r)^{2} \\
        & =2 \psi_{n}\sum_{j=1}^{n-1} h_{j}^{2} \theta_{j}-\left(\sum_{i} h_{i} \theta_{i}\right)^{2} =2 \psi_{n}\sum_{j=1}^{n-1} h_{j}^{2} \theta_{j}-\sum_{i j=1}^{n-2} h_{i} h_{j} \theta_{i} \theta_{j} \\
        &=2 \psi_{n}\sum_{j=1}^{n-1} h_{j}^{2} \theta_{j}-\sum_{i j=1}^{n-2} h_{i} h_{j} \theta_{i} \theta_{j}\leq 2 \sum_{j=1}^{n-1} h_{j}^{2}\delta^2-\sum_{i j} h_{i} h_{j} \varepsilon^{2} \\
        & =\sum_{j=1}^{n-1} h_{j}^{2}\left(2\delta^2-\varepsilon^{2}\right)-\sum_{i\neq j} h_{i} h_{j} \varepsilon^{2} \\
        & \leq (2\delta^2-\varepsilon^{2})h_{n-1}^{2} + (2\delta^2-\varepsilon^{2}) h_{n-2}h_{n-1}+\cdots+(2\delta^2-\varepsilon^{2})h_{1}h_{2} - 2\varepsilon^{2} \sum_{i<j} h_{i} h_{j}
    \end{split}
    \end{equation}
    so
    \begin{equation}
    \nonumber
    \begin{split}
       &2\psi_{n}\langle\nabla^{2}{r}, \nabla^{2}{r}\circ T\rangle-(\square r)^{2}\\
        & \leq(2\delta^2-\varepsilon^{2})h_{n-1}^{2} + (2\delta^2-3\varepsilon^{2} )h_{n-2}h_{n-1}+\ldots+(2\delta^2-3\varepsilon^{2})h_{1}h_{2} - 2\varepsilon^{2} \!\!\!\sum_{\substack{i<j \\
        j\neq i+1}} \!\!\! h_{i} h_{j} \\
        & \leq (2\delta^2-\varepsilon^{2})h_{n-1}^{2}+ (n-2)(2\delta^2-2\varepsilon^{2})h_{n-1}^{2}-(n-2)\varepsilon^{2} h_{1}^{2}-[(n-2)^{2}-(n-2)] \varepsilon^{2}h_{1}^{2} \\[10pt]
        & =[2(n-1)\delta^2-(2n - 3)\varepsilon^{2}] h_{n-1}^{2}-(n-2)^{2}\varepsilon^{2}h_{1}^{2},
    \end{split}
    \end{equation}
    and then
    \begin{equation}\label{4.12}
        2\psi_{n}\langle\nabla^{2}{r}, \nabla^{2}{r}\circ T\rangle-(\square r)^{2}\leq [2(n-1)\delta^2-(2n - 3)\varepsilon^{2}] h_{n-1}^{2}-(n-2)^{2}\varepsilon^{2}h_{1}^{2}.
    \end{equation}
 Recall that $-\kappa_{1}^{2}\leq K\leq-\kappa_{2}^{2}$, so 
    \begin{equation}
    \nonumber
    \begin{split}
        2\psi_{n}R_{\eta,T}(\partial_{r},\partial_{r})&=2\psi_{n}\!\left(\sum_{j=1}^{n-1}\langle{R}(e_{j},\partial_{r})\partial_{r},T(e_{j})\rangle-\langle\nabla_{\partial_{r}}((\mathrm{div}{T})^{\sharp}-T(\nabla{\eta})),\partial_{r}\rangle\!\right) \\
        &=2\psi_{n}\left(\sum_{j=1}^{n-1}\psi_{j}\langle{R}(e_{j},\partial_{r})\partial_{r},e_{j}\rangle+\langle\nabla_{\partial_{r}}T(\nabla{\eta}),\partial_{r}\rangle\right)\\
        &\leq2\psi_{n}\left(\sum_{j=1}^{n-1}\psi_{j}\left(-k_{2}^{2}\right)+\langle(\nabla_{\partial_{r}}T)(\nabla{\eta})+T(\nabla_{\partial_{r}}\nabla{\eta}),\partial_{r}\rangle\right),
    \end{split}
	\end{equation}
    and once that $\nabla_{\partial_{r}}T=0$
    \begin{equation}\label{4.13}
    \begin{split}
        2\psi_{n}R_{\eta,T}(\partial_{r},\partial_{r})&\leq2\psi_{n}\left(\sum_{j=1}^{n-1}\psi_{j}\left(-k_{2}^{2}\right)+\langle \nabla_{\partial_{r}}\nabla{\eta},T(\partial_{r})\rangle\right)\\
        &=-2\sum_{j=1}^{n-1}\psi_{n}\psi_{j}k_{2}^{2}+2\psi_{n}^2\langle\nabla_{\partial_{r}}\nabla{\eta},\partial_{r}\rangle\\
        &\leq-2(n-1)\varepsilon^{2}k_{2}^{2}+2\psi_{n}^2|\nabla^{2}\eta(\partial_{r},\partial_{r})|\\
        &\leq-2(n-1)\varepsilon^{2}k_{2}^{2}+2\delta^2\eta_{1},
    \end{split}
	\end{equation}
    where $\eta_{1} = \sup_{\Bar{\Omega}} |\nabla^{2}\eta(\partial_{r},\partial_{r})|$.

    Using Rauch Comparison Theorem (Lemma \ref{RC}), we get 
    \begin{equation}
        \nonumber
        \kappa_{1}\frac{\cosh(\kappa_{1}r)}{\sinh(\kappa_{1}r)}\geq h_{n-1}\geq\cdots\geq h_{1}\geq\kappa_{2}\frac{\cosh(\kappa_{2}r)}{\sinh(\kappa_{2}r)},
    \end{equation}
    then, from \eqref{4.11} \eqref{4.12} e \eqref{4.13} we have
    \begin{equation}
    \nonumber
    \begin{split}
        -&(\square_{\eta}{r})^{2} - 2\langle\nabla(\square_{\eta}{r}),\!T(\partial_{r})\rangle \\
        &= 2\psi_{n}\langle\nabla^{2}{r}, \nabla^{2}{r}\circ T\rangle-\left(\square{r}\right)^{2}+2\langle\nabla\eta, T(\partial_{r})\rangle\square{r}-\langle\nabla\eta, T(\partial_{r})\rangle^{2} + 2\psi_{n}R_{\eta,T}(\partial_{r},\partial_{r}) \\
        &\leq [2(n-1)\delta^{2}-(2n-3)\varepsilon^{2}]h_{n-1}^{2} - (n-2)^{2}\varepsilon^{2}h_{1}^{2} - 2(n-1)\varepsilon^{2}\kappa_{2}^{2} \\
        &\quad \ + 2\delta^{2}\eta_{1} + 2\left\langle\nabla\eta, T(\partial_{r})\right\rangle\square r.
    \end{split}
    \end{equation}
    From Lemma \ref{RC} we have
    \begin{equation}
    \nonumber
    \begin{split}
        -&(\square_{\eta}{r})^{2} - 2\langle\nabla(\square_{\eta}{r}),\!T(\partial_{r})\rangle \\
        &\leq [2(n-1)\delta^{2}-(2n-3)\varepsilon^{2}]\kappa_{1}^{2}\frac{\cosh^{2}(\kappa_{1}r)}{\sinh^{2}(\kappa_{1}r)} - (n-2)^{2}\varepsilon^{2}\kappa_{2}^{2}\frac{\cosh^{2}(\kappa_{2}r)}{\sinh^{2}(\kappa_{2}r)} \\
        &\quad \ - 2(n-1)\varepsilon^{2}\kappa_{2}^{2} + 2\delta^{2}\eta_{1} + 2\left\langle\nabla\eta, T(\partial_{r})\right\rangle\square r \\[10pt]
        &=[2(n-1)\delta^{2}-(2n-3)\varepsilon^{2}]\left(\kappa_{1}^{2} + \frac{\kappa_{1}^{2}}{\sinh^{2}(\kappa_{1}r)}\right) - (n-2)^{2}\varepsilon^{2}\left(\kappa_{2}^{2} + \frac{\kappa_{2}^{2}}{\sinh^{2}(\kappa_{2}r)}\right) \\
        &\quad \ - 2(n-1)\varepsilon^{2}\kappa_{2}^{2} + 2\delta^{2}\eta_{1} + 2\left\langle\nabla\eta, T(\partial_{r})\right\rangle\square r.
    \end{split}
    \end{equation}
    Consequently,
    \begin{equation}\label{4.14}
    \begin{split}
        -&(\square_{\eta}{r})^{2} - 2\langle\nabla(\square_{\eta}{r}),\!T(\partial_{r})\rangle \\
        &\leq [2(n-1)\delta^{2}-(2n-3)\varepsilon^{2}]\kappa_{1}^{2} - [n^2-2n+2] \varepsilon^{2}\kappa_{2}^{2} + 2\delta^{2}\eta_{1} + 2\left\langle\nabla\eta, T(\partial_{r})\right\rangle\square r \\
        &\quad \ + [2(n-1)\delta^{2}-(2n-3)\varepsilon^{2}]\frac{\kappa_{1}^{2}}{\sinh^{2}(\kappa_{1}r)} - (n-2)^{2}\varepsilon^{2}\frac{\kappa_{2}^{2}}{\sinh^{2}(\kappa_{2}r)}.
    \end{split}
    \end{equation}
    Here, there are three cases to consider:
    \begin{itemize}
	\item[(1)] $0<k_{2}\leq{k_{1}}$: since $0<r$ and $\displaystyle{f(\theta)}=\frac{\theta^2}{\sinh^{2}(\theta r)}$ is a decreasing function for $\theta>0$, we have
	\begin{equation}
    \nonumber
		\frac{k_{1}^2}{\sinh^{2}(k_{1}r)}\leq\frac{k_{2}^2}{\sinh^{2}(k_{2}r)},
	\end{equation}
then,
    \begin{equation}
    \nonumber
    \begin{split}
        [&2(n-1)\delta^{2}-(2n-3)\varepsilon^{2}]\frac{\kappa_{1}^{2}}{\sinh^{2}(\kappa_{1}r)} - (n-2)^{2}\varepsilon^{2}\frac{\kappa_{2}^{2}}{\sinh^{2}(\kappa_{2}r)} \\
        &\leq [2(n-1)\delta^{2}-(2n-3)\varepsilon^{2} - (n-2)^{2}\varepsilon^{2}]\frac{\kappa_{1}^{2}}{\sinh^{2}(\kappa_{1}r)} \\
        &= [2(n-1)\delta^{2} -(n-1)^{2}\varepsilon^{2}]\frac{\kappa_{1}^{2}}{\sinh^{2}(\kappa_{1}r)}.
    \end{split}
    \end{equation}
    Therefore,
    \begin{equation}\label{4.15}
    \begin{split}
        \qquad\quad-(\square_{\eta}&{r})^{2} - 2\langle\nabla(\square_{\eta}{r}),\!T(\partial_{r})\rangle \\
        &\leq [2(n-1)\delta^{2}-(2n-3)\varepsilon^{2}]\kappa_{1}^{2} - [n^2-2n+2] \varepsilon^{2}\kappa_{2}^{2} + 2\delta^{2}\eta_{1} \\
        &\quad \ + 2\left\langle\nabla\eta, T(\partial_{r})\right\rangle\square r + [2(n-1)\delta^{2} -(n-1)^{2}\varepsilon^{2}]\frac{\kappa_{1}^{2}}{\sinh^{2}(\kappa_{1}r)}.
    \end{split}
    \end{equation}

    \item[(2)] $0=k_{2}<{k_{1}}$: in this case, we begin by estimating the expression
	\begin{equation}
    \nonumber
    \begin{split}
        -(\square_{\eta}&{r})^{2} - 2\langle\nabla(\square_{\eta}{r}),\!T(\partial_{r})\rangle \\
        &= [2(n-1)\delta^{2}-(2n-3)\varepsilon^{2}]\kappa_{1}^{2} + 2\delta^{2}\eta_{1} + 2\left\langle\nabla\eta, T(\partial_{r})\right\rangle\square r \\
        &\quad \ + [2(n-1)\delta^{2}-(2n-3)\varepsilon^{2}]\frac{\kappa_{1}^{2}}{\sinh^{2}(\kappa_{1}r)} - (n-2)^{2}\varepsilon^{2}\frac{1}{r^2}.
    \end{split}
    \end{equation}
    Since $0<{k_{1}}$ e $0<r$, we know that
	\begin{equation}
    \nonumber
        \frac{k_{1}^2}{\sinh^{2}(k_{1}r)}\leq\frac{1}{r^{2}},
	\end{equation}
    because it is sufficient to check that $\displaystyle{g(\theta)}=\sinh(a\theta)-a\theta$ is an increasing function for $a>0$ and $\theta>0$,  with $g(0)=0$. Then,
	\begin{equation}
    \nonumber
    \begin{split}
        [&2(n-1)\delta^{2}-(2n-3)\varepsilon^{2}]\frac{\kappa_{1}^{2}}{\sinh^{2}(\kappa_{1}r)} - (n-2)^{2}\varepsilon^{2}\frac{1}{r^2} \\
        &\leq [2(n-1)\delta^{2} -(n-1)^{2}\varepsilon^{2}]\frac{\kappa_{1}^{2}}{\sinh^{2}(\kappa_{1}r)}.
    \end{split}
	\end{equation}
	Thus,
	\begin{equation}\label{4.16}
        \begin{split}
        -(\square_{\eta}&{r})^{2} - 2\langle\nabla(\square_{\eta}{r}),\!T(\partial_{r})\rangle \\
        &\leq [2(n-1)\delta^{2}-(2n-3)\varepsilon^{2}]\kappa_{1}^{2} + 2\delta^{2}\eta_{1} + 2\left\langle\nabla\eta, T(\partial_{r})\right\rangle\square r \\
        &\quad \ + [2(n-1)\delta^{2} -(n-1)^{2}\varepsilon^{2}]\frac{\kappa_{1}^{2}}{\sinh^{2}(\kappa_{1}r)}.
        \end{split}
    \end{equation}

    \item[(3)] $0=k_{2}={k_{1}}$: here, we begin by estimating the expression
	\begin{equation}\label{4.17}
    \begin{split}
	    -(\square_{\eta}&{r})^{2} - 2\langle\nabla(\square_{\eta}{r}),\!T(\partial_{r})\rangle \\
        &\leq 2\delta^{2}\eta_{1} + 2\left\langle\nabla\eta, T(\partial_{r})\right\rangle\square r + [2(n-1)\delta^{2} -(n-1)^{2}\varepsilon^{2}]\frac{1}{r}.
    \end{split}
	\end{equation}
    \end{itemize}

    From inequalities \eqref{4.15}, \eqref{4.16} and \eqref{4.17}  the inequality \eqref{4.14} becomes
    \begin{equation}
    \nonumber
    \begin{split}
        -(\square_{\eta}{r}&)^{2} - 2\langle\nabla(\square_{\eta}{r}),\!T(\partial_{r})\rangle \\
        \leq& [2(n-1)\delta^{2}-(2n-3)\varepsilon^{2}]\kappa_{1}^{2} - [n^2-2n+2] \varepsilon^{2}\kappa_{2}^{2} + 2\delta^{2}\eta_{1} + 2\left\langle\nabla\eta, T(\partial_{r})\right\rangle\square r,
    \end{split}
    \end{equation}
    when $\Big[2(n-1)\delta^{2}-(n-1)^{2}\varepsilon^{2}\Big]\leq0$.

    From inequalities \eqref{4.15}-\eqref{4.17}  the inequality \eqref{4.14} becomes
    \begin{equation}
    \nonumber
    \begin{split}
        -(\square_{\eta}&{r})^{2} - 2\langle\nabla(\square_{\eta}{r}),\!T(\partial_{r})\rangle \\
        &\leq [2(n-1)\delta^{2}-(2n-3)\varepsilon^{2}]\kappa_{1}^{2} - [n^2-2n+2] \varepsilon^{2}\kappa_{2}^{2} + 2\delta^{2}\eta_{1} \\ 
        &\quad \ + 2\left\langle\nabla\eta, T(\partial_{r})\right\rangle\square r + \Big[2(n-1)\delta^{2}-(n-1)^{2}\varepsilon^{2}\Big]\frac{1}{d^{2}},
    \end{split}
	\end{equation}
    when $\Big[2(n-1)\delta^{2}-(n-1)^{2}\varepsilon^{2}\Big]>0$, since by taking $d=dist(\Omega,o)$, we have $0<d\leq{r}(x)$ and $\frac{1}{r^{2}}\leq\frac{1}{d^{2}}$  for all $x\in\Omega$.

	Let us define 
	\begin{equation}
        \nonumber
		a(n,T):=\left\{ \begin{array}{cl}
		0, & \mbox{ se  }\ 2(n-1)\delta^{2}-(n-1)^{2}\varepsilon^{2}\leq0, \\[5pt]
		2(n-1)\delta^{2}-(n-1)^{2}\varepsilon^{2}, & \mbox{ se }\ 2(n-1)\delta^{2}-(n-1)^{2}\varepsilon^{2} >0,
		\end{array}\right.
	\end{equation}
	so that
    \begin{equation} \label{4.18}
    \begin{split}
	    -(\square_{\eta}{r})^{2} - 2\langle\nabla(\square_{\eta}{r}),\!T(\partial_{r})\rangle
        &\leq [2(n-1)\delta^{2}-(2n-3)\varepsilon^{2}]\kappa_{1}^{2} - [n^2-2n+2] \varepsilon^{2}\kappa_{2}^{2} \\ 
        &\quad \ + 2\delta^{2}\eta_{1} + 2\left\langle\nabla\eta, T(\partial_{r})\right\rangle\square r + \frac{a(n,T)}{d^{2}}.
    \end{split}
	\end{equation}

    By using \cite[Proposition~3, p. 11]{FG} by Fonseca and Gomes, we get 
    \begin{equation}\label{4.19}
        \int_\Omega u_1^2 \square r \langle \nabla \eta , T(\partial_{r})\rangle dm \leq \delta^2\eta_{r}(n-1)\left(\kappa_1 + \frac{1}{d}\right),
    \end{equation}
    where $\eta_{r}=\max_{\overline{\Omega}}|\langle \nabla\eta,\partial_{r}\rangle|$, then from \eqref{4.10}, \eqref{4.18} and \eqref{4.19}  we obtain
	\begin{equation}\label{4.20}
    \begin{split}
        \lambda_{k+2}-\lambda_{k+1} \leq&\frac{4}{\sqrt{\sigma}}\left(\delta\lambda_{1}+\frac{[2(n-1)\delta^{2}-(2n-3)\varepsilon^{2}]\kappa_{1}^{2} - [n^2-2n+2] \varepsilon^{2}\kappa_{2}^{2}+2\delta^{2}\eta_{1}}{4}\right.\\
        &\qquad \quad \left.+\frac{\delta^2 \eta_{r}(n-1)(\kappa_1 + \frac{1}{d})}{2}+ \frac{a(n,T)}{4d^{2}}\right)^\frac{1}{2}\sqrt{\lambda_{k+2}}.
	\end{split}
    \end{equation}
	From Lemma~\ref{Yang-Ineq1}, for the $\square_{\eta}$ operator, we have
	\begin{equation}\label{4.21}
        \lambda_{k+2}\leq\left(1+\frac{4\delta}{n\varepsilon}\right)\left(\lambda_{1}+\frac{n^{2}H_{0}^{2}+4C_{0}}{4\delta}\right)(k+1)^{\frac{2\delta}{n\varepsilon}},
	\end{equation}
	where $\displaystyle C_{0} = \sup_{\overline{\Omega}}\left\{\frac{1}{2}\mathrm{div}
	\big(T^{2}(\nabla\eta)\big) - \frac{1}{4}|T(\nabla\eta)|^{2} \right\}$ and $\displaystyle{H_{0}}=\sup_{\overline{\Omega}}|{\bf{H}}_{T}|$. From \eqref{4.20} and \eqref{4.21} we obtain
	\begin{equation}\label{4.22}
    \begin{split}
		\lambda_{k+2}&-\lambda_{k+1}\\
        &\leq\frac{4}{\sqrt{\sigma}}\left(\delta\lambda_{1}\!+\!\frac{[2(n-1)\delta^{2}-(2n-3) \varepsilon^{2}]\kappa_{1}^{2} - [n^2-2n+2] \varepsilon^{2}\kappa_{2}^{2}+2\delta^{2}\eta_{1}}{4}\right.\\
		&\left. +\frac{\delta^2\eta_{r}(n-1)(\kappa_1 + \frac{1}{d})}{2}+ \frac{a(n,T)}{4d^{2}}\right)^\frac{1}{2}\!\left(\!\left(1+\frac{4\delta}{n\varepsilon}\right)\!\left(\lambda_{1}\!+\!\frac{n^{2}H_{0}^{2}\!+\!4C_{0}}{4\delta}\!\right)\!(k+1)^{\frac{2\delta}{n\varepsilon}}\!\right)^\frac{1}{2}\\
		=&\frac{4}{\sqrt{\sigma}}\left(\delta\lambda_{1}\!+\!\frac{[2(n-1)\delta^{2}-(2n-3) \varepsilon^{2}]\kappa_{1}^{2} - [n^2-2n+2] \varepsilon^{2}\kappa_{2}^{2}+2\delta^{2}\eta_{1}}{4}\right.\\
		&\left.+\frac{\delta^2\eta_{r}(n-1)(\kappa_1 + \frac{1}{d})}{2}+ \frac{a(n,T)}{4d^{2}}\right)^{\!\!\frac{1}{2}}\!\!\!\left(1+\frac{4\delta}{n\varepsilon}\right)^{\!\!\frac{1}{2}}\!\!\left(\lambda_{1}\!+\!\frac{n^{2}H_{0}^{2}+4C_{0}}{4\delta}\right)^{\!\!\frac{1}{2}}(k+1)^{\frac{\delta}{n\varepsilon}}\\
		=& C_{n,\Omega}(k+1)^{\frac{\delta}{n\varepsilon}},
	  \end{split}
    \end{equation}
	where
	\begin{equation}
    \nonumber
    \begin{split}
        C_{n,\Omega}=&\frac{4}{\sqrt{\sigma}}\left(\delta\lambda_{1}\!+\!\frac{[2(n-1)\delta^{2}-(2n-3) \varepsilon^{2}]\kappa_{1}^{2} - [n^2-2n+2] \varepsilon^{2}\kappa_{2}^{2}+2\delta^{2}\eta_{1}}{4}\right.\\
        &\left.+\frac{\delta^2\eta_{r}(n-1)(\kappa_1 + \frac{1}{d})}{2}+ \frac{a(n,T)}{4d^{2}}\right)^{\!\!\frac{1}{2}}\!\!\!\left(1+\frac{4\delta}{n\varepsilon}\right)^{\!\!\frac{1}{2}}\!\!\left(\lambda_{1}\!+\!\frac{n^{2}H_{0}^{2}+4C_{0}}{4\delta}\right)^{\!\!\frac{1}{2}}.
    \end{split}
	\end{equation}
	Consequently, inequality \eqref{4.22} is valid for all $k>1$.
\end{proof}

\end{document}